\documentclass{amsart}
\usepackage{amsfonts}
\usepackage{amssymb}
\usepackage{amsmath}
\usepackage{amsthm}
\usepackage[pdftex]{graphicx}

\usepackage{moreverb}
\usepackage{fancybox}
\usepackage{fancyvrb}
\usepackage{comment}
\usepackage{color}
\usepackage{multirow}
\usepackage[all]{xy}
\usepackage{scalefnt}
\usepackage{stmaryrd}

\theoremstyle{plain}
\newtheorem{theorem}{Theorem}[section]
\newtheorem{lemma}[theorem]{Lemma}
\newtheorem{cor}[theorem]{Corollary}
\newtheorem{prop}[theorem]{Proposition}

\newtheorem{question}{Question}

\newtheorem{obs}[theorem]{Observation}

\theoremstyle{definition}
\newtheorem{definition}[theorem]{Definition}
\newtheorem{example}[theorem]{Example}

\theoremstyle{definition}
\newtheorem*{remark}{Remark}
\newtheorem{claim}{Claim}
\newtheorem*{notation}{Notation}

\newtheorem*{ack}{Acknowledgements}

\newcommand{\upto}{\upharpoonright}
\newcommand{\fr}{\mbox{}^\smallfrown}
\newcommand{\N}{\mathbb{N}}
\newcommand{\om}{\omega}
\newcommand{\pcolon}{\colon\!\!\!\subseteq}
\newcommand{\ep}{\varepsilon}

\makeatletter
\providecommand{\bignplus}{%
  \mathop{%
    \mathpalette\@updown\biguplus
  }%
}
\newcommand*{\@updown}[2]{%
  \rotatebox[origin=c]{180}{$\m@th#1#2$}%
}
\makeatother

\newdir{ >}{{}*!/-10pt/@{>}} 

\newcommand{\pushoutcorner}[1][dr]{\save*!/#1-1.2pc/#1:(-1,1)@^{|-}\restore}

\newcommand{\pair}[1]{\langle #1 \rangle}
\newcommand{\wunion}{\uplus}

\newcommand{\arr}{\rightarrowtriangle}

\newcommand{\inj}{\rightarrowtail}

\newcommand{\mono}{\rightarrowtail}
\newcommand{\embed}{\hookrightarrow}

\newcommand{\edge}[1]{\overset{#1}{\longrightarrow}}

\newcommand{\monoedge}[1]{\overset{#1}{\inj}}

\newcommand{\tto}{\rightrightarrows}

\newcommand{\eval}[1]{\llbracket #1 \rrbracket}

\title{Many-one reducibility with realizability}
\author{Takayuki Kihara}
\date{}

\begin{document}
\maketitle

\begin{abstract}
In this article, we propose a new classification of $\Sigma^0_2$ formulas under the realizability interpretation of many-one reducibility (i.e., Levin reducibility).
For example, ${\sf Fin}$, the decision of being eventually zero for sequences, is many-one/Levin complete among $\Sigma^0_2$ formulas of the form $\exists n\forall m\geq n.\varphi(m,x)$, where $\varphi$ is decidable.
The decision of boundedness for sequences ${\sf BddSeq}$ and for width of posets ${\sf FinWidth}$ are many-one/Levin complete among $\Sigma^0_2$ formulas of the form $\exists n\forall m\geq n\forall k.\varphi(m,k,x)$, where $\varphi$ is decidable.
However, unlike the classical many-one reducibility, none of the above is $\Sigma^0_2$-complete.
The decision of non-density of linear orders ${\sf NonDense}$ is truly $\Sigma^0_2$-complete.
\end{abstract}

\section{Introduction}

In this article, we introduce the notion of many-one reducibility for {\em sets with witnesses} and reanalyze the arithmetical/Borel/projective hierarchy under this new reducibility notion.
In computational complexity theory, the notion of polytime many-one reducibility for sets with witnesses (a.k.a.~search problems or function problems) is known as {\em Levin reducibility} \cite{Lev73}, but strangely enough, it seems that its computable analogue has never been studied.

\begin{definition}[Levin \cite{Lev73}]\label{def:Levin-reducibility}
Let $\Sigma$ be a finite alphabet.
A search problem (or a set with witnesses) is a binary relation $R\subseteq\Sigma^\ast\times\Sigma^\ast$, and any $y$ satisfying $R(x,y)$ is called a witness (or a certificate) for $x\in|R|$, where $|R|=\{x:\exists yR(x,y)\}$.
For a complexity class $\mathcal{C}$ and search problems $A$ and $B$, we say that $A$ is {\em $\mathcal{C}$-Levin reducible to $B$} if there exist $\mathcal{C}$-functions $\varphi,r_-,r_+$ such that for any $x,y,z\in\Sigma^\ast$ the following holds:
\begin{enumerate}
\item $x\in |A|$ if and only if $\varphi(x)\in |B|$.
\item If $y$ is a witness for $x\in|A|$ then $r_-(x,y)$ is a witness for $\varphi(x)\in |B|$.
\item If $z$ is a witness for $\varphi(x)\in|B|$ then $r_+(x,z)$ is a witness for $x\in |A|$.
\end{enumerate}
\end{definition}

As a closer look at the definition shows, Levin reducibility is nothing more than the realizability interpretation of many-one reducibility.
In this article, we introduce the notion of many-one reducibility for subobjects in any category having pullbacks, and observe that the same definition as computable Levin reducibility is restored as many-one reducibility in the category of represented spaces.
This perspective unexpectedly connects the notion of Levin reducibility with the study of arithmetical/Borel/projective hierarchy in intuitionistic systems.

The notion of many-one reducibility in intuitionistic/constructive systems was first studied exhaustively by Veldman \cite{VeldmanPhD,Vel90,Vel09,Vel22} and later more recently by \cite{FoFe23} and others.
According to Veldman, the intuitionistic Borel/projective hierarchy behaves differently from the classical hierarchy.
For instance, Veldman showed that, under a certain intuitionistic system, the union of two $\Pi^0_1$ sets is not necessarily $\Pi^0_1$ \cite{VeldmanPhD}, the set ${\sf Fin}$ of all sequences which is eventually zero is not $\Sigma^0_2$-complete \cite{Vel09}, and the set ${\sf IFKB}$ of all trees which is ill-founded w.r.t.~the Kleene-Brouwer ordering is not $\Sigma^1_1$-complete \cite{Vel22}.

We see that these seemingly strange results can be clearly understood using Levin reducibility.
The category of represented spaces has the natural numbers object $\om$, the exponential object $\om^\om$, and the interpretation of first-order logic, so one can introduce the notion of arithmetic/Borel/projective subobjects by interpreting their defining formulas in the internal logic.
Moreover, in this category, a subobject is nothing but a subset with witnesses.
Based on these observations, for instance, one can understand Veldman's result as meaning that the witnessed version of ${\sf Fin}$ is simply not Levin-complete among the witnessed $\Sigma^0_2$ subsets (in classical mathematics).
The same applies to ${\sf IFKB}$.
In this way, even classical mathematicians can clearly understand Veldman's results on the intuitionistic hierarchy.

Of course, merely giving an interpretation of the existing results is not very interesting, so we push this point of view forward with further analysis of many-one/Levin degrees of witnessed sets.
In this article, we focus in particular on $\Sigma^0_2$ sets with their existential witnesses.
Our results are summarized in Figure \ref{fig:Sigma2}.
\begin{figure}[t]
\includegraphics[width=80mm]{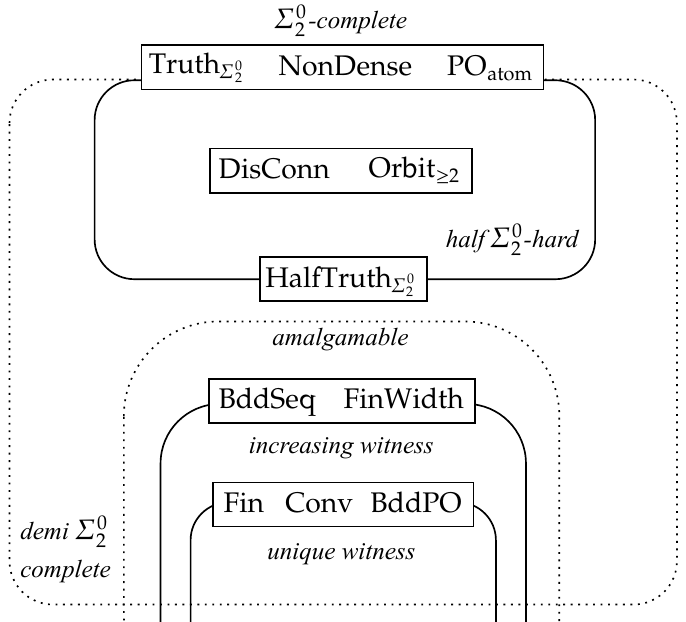}
\caption{The fine analysis of classical $\Sigma^0_2$-complete sets}\label{fig:Sigma2}
\end{figure}

For example, ${\sf Fin}$ is many-one/Levin complete among $\Sigma^0_2$ sets defined by formulas of the form $\exists n\forall m\geq n.\varphi(m,x)$, where $\varphi$ is decidable.
The decision of boundedness for posets ${\sf BddPO}$ is also at the same level. 
The decision of boundedness for sequences ${\sf BddSeq}$ is many-one/Levin complete among $\Sigma^0_2$ sets defined by formulas of the form $\exists n\forall m\geq n\forall k.\varphi(m,k,x)$, where $\varphi$ is decidable.
The decision of finiteness for width ${\sf FinWidth}$ and height ${\sf FinHeight}$ for posets are also at the same level. 
The decision of non-density of linear orders ${\sf NonDense}$ is truly $\Sigma^0_2$-complete.
In this way, the focus on (existential) witnesses leads us to the discovery of a previously unknown classification of $\Sigma^0_2$ sets.

\section{Preliminaries}

\subsection{Represented space}
A coding system is a set ${\tt Code}$ of symbols for coding various mathematical objects, with a prior specification of which functions on ${\tt Code}$ are realizable.
There are three typical coding systems:
\begin{enumerate}
\item Kleene's first algebra ${\sf K}_1$: ${\tt Code}=\om$, and
\begin{center}
realizable functions $=$ computable functions on $\om$.
\end{center}
\item Kleene's second algebra ${\sf K}_2$: ${\tt Code}=\om^\om$, and
\begin{center}
realizable functions $=$ continuous functions on $\om^\om$.
\end{center}
\item Kleene-Vesley algebra ${\sf KV}$: ${\tt Code}=\om^\om$, and
\begin{center}
realizable functions $=$ computable functions on $\om^\om$.
\end{center}
\end{enumerate}

Here, $\om$ denotes the set of all natural numbers.

\begin{notation}
We write $p\ast x$ as the output result of feeding an input $x$ to the realizable function coded by $p$.
For instance, $e\ast x$ in ${\sf K}_1$ stands for $\{e\}(x)$ or $\varphi_e(x)$ in traditional notation.
\end{notation}

Hereafter we assume that a coding system ${\tt Code}$ is one of ${\sf K}_1$, ${\sf K}_2$, or ${\sf KV}$.
Of course, it is obviously possible to consider an arbitrary {\em relative partial combinatory algebra} (or more) as a coding system \cite{vOBook}.
However, in order to lower the threshold for reading, we avoid unnecessary generalizations as much as possible.

\begin{definition}\label{def:rep-sp-intro}
A {\em represented space} $X$ consists of an underlying set $|X|$ and a partial surjection $\delta_X\pcolon{\tt Code}\to|X|$.
We sometimes use the symbol $p\vdash_Xx$ to denote $\delta_X(p)=x$, and say that $p$ is an $X$-name of $x$ or $p$ is a name of $x\in X$.
\end{definition}

We sometimes use $E_X(x)$ to denote the set of all names of $x\in X$.

\begin{example}\label{exa:code-itself}
${\tt Code}$ itself is a represented space via the identity map ${\rm id}\colon{\tt Code}\to{\tt Code}$. 
\end{example}

\begin{example}\label{exa:terminal}
The {\em terminal space} $\mathbf{1}$ is defined as follows:
The underlying set is $|\mathbf{1}|=\{\bullet\}$, and any $p\in{\tt Code}$ is a name of its unique element $\bullet$.
\end{example}

\begin{example}
The space of natural numbers ${\sf Nat}$ is defined as follows:
The underlying set is $|\mathsf{Nat}|=\om$.
In ${\sf K}_1$, $n\in\om$ is a name of $n\in{\sf Nat}$.
In ${\sf K}_2$ or ${\sf KV}$, $n0^\infty\in\om^\om$ is a name of $n\in{\sf Nat}$, where $n0^\infty$ is the infinite string resulting from concatenating $n$ followed by the zero sequence $0^\infty$; that is, $(n0^\infty)(0)=n$ and $(n0^\infty)(k)=0$ for any $k>0$.
By an abuse of notation, we often use $\om$ to denote ${\sf Nat}$.
\end{example}

\begin{example}\label{exa:Sierpinski-space}
The {\em represented Sierpi\'nski space} $\mathbb{S}$ is defined as follows:
\begin{itemize}
\item The underlying set is $|\mathbb{S}|=\{\top,\bot\}$.
\item (${\sf K}_1$)
If $p\ast 0\downarrow$ then $p$ is a name of $\top$ else $p$ is a name of $\bot$.
\item (${\sf K}_2$ or ${\sf KV}$)
The zero sequence $0^\infty$ is a name of $\bot$ and the other sequences are names of $\top$.
\end{itemize}

See also Bauer \cite{Bau00} for a common construction of the Sierpi\'nski space (a.k.a.~the Rosolini dominance) in general coding systems.
\end{example}

\begin{definition}
Let $X$ and $Y$ be represented spaces.
A {\em morphism $f\colon X\to Y$} is a function $f\colon |X|\to |Y|$ such that there exists a partial realizable function $F$ on ${\tt Code}$ such that if $p$ is a name of $x\in X$ then $F(p)$ is a name of $f(x)\in Y$.
We often call $F$ a {\em tracker} of $f$.
\end{definition}

\begin{definition}
For represented spaces $X,Y$, a morphism $f\colon X\to Y$ is {\em mono} if it is injective on underlying sets.
\end{definition}

\section{The structure of subobjects}

\subsection{Witnessed subset}
Our aim is to consider an arithmetic hierarchy over a represented space.
For this purpose, we carefully consider what a subset of a represented space is.

\begin{definition}
A {\em subspace} of a represented space $X$ is a represented space $A$ such that $|A|\subseteq |X|$ and $\delta_A=\delta_X{\upto_{|A|}}$; that is, the $A$- and $X$-names of $x\in |A|$ are the same.
\end{definition}

The notion of a subspace seems most appropriate when viewing a subset of a represented space as a represented space again.
However, there could be another possibility.

\begin{definition}\label{def:subobject-regular}
A {\em subobject} of a represented space $X$ is a represented space $A$ such that $|A|\subseteq |X|$ and there exists a partial realizable function which, given an $A$-name of $x\in|A|$, returns its $X$-name.

A {\em regular subobject} of a represented space $X$ is a subobject $A$ of $X$ such that there exists a partial realizable function which, given an $X$-name of $x\in |A|$, returns its $A$-name.
\end{definition}

\begin{obs}
A represented space $A$ is a subobject of $X$ iff $|A|\subseteq|X|$ and the inclusion map $i\colon A\mono X$ is a morphism.

A represented space $A$ is a regular subobject of $X$ iff it is a subobject of $X$ and the inclusion morphism $i\colon A\mono X$ has a partial inverse morphism $i^{-1}\pcolon X\to A$; that is, $i^{-1}(i(x))=x$ for any $x\in |A|$.
\end{obs}

The most basic relation between subsets is the inclusion relation.
We introduce the inclusion relation between subobjects as follows:

\begin{definition}
For subobjects $A,B$ of a represented space $X$, we say that $A$ is {\em included in $B$} (written $A\subseteq B$) if $|A|\subseteq|B|$ and the inclusion map $i\colon A\embed B$ is a morphism.
If $A\subseteq B$ and $B\subseteq A$ we say that $A$ is equivalent to $B$ and write $A\equiv B$.
\end{definition}

\begin{obs}
A subspace of $X$ is a regular subobject of $X$.
Conversely, every regular subobject of $X$ is equivalent to a subspace of $X$.
\end{obs}

Thus, one can understand that a regular subobject is a represented space obtained by taking a subset of a represented space.
What then is the value of non-regular subobjects?
To answer this, it is better to think of the notion of a (non-regular) subobject as a subset with additional information, such as a {\em subset with witnesses}, rather than just a subset.

\begin{definition}
A {\em witnessed subset} $A$ of a represented space $X$ is a represented space such that $|A|\subseteq|X|$ and every name of $x\in A$ is a pair $\langle w,p\rangle$ of an $X$-name $p$ of $x$ and some $w\in{\tt Code}$.
In this case, $w$ is called a {\em witness} for $x\in A$.
\end{definition}

One can see that a subobject of a represented space is nothing more than a witnessed subset:

\begin{obs}
A witnessed subset of $X$ is a subobject of $X$.
Conversely, every subobject of $X$ is equivalent to a witnessed subset of $X$.
\end{obs}

\begin{proof}
For a witnessed subset $A$ of $X$, the inclusion map $i\colon A\embed X$ is clearly tracked by the projection map $\pi_1\colon\langle w,p\rangle\mapsto p$.
Conversely, if $A$ is a subobject of $X$, that is, $i\colon A\embed X$ is tracked by some $f$, then consider the following represented space $A_f$:
The underlying set is $|A_f|=|A|$, and $\langle w,p\rangle$ is a name of $x\in A_f$ iff $w$ is a name of $x\in A$ and $p=f(w)$.
Then, $A_f$ is clearly a witnessed subset of $X$.
Moreover, $A_f\subseteq A$ is tracked by $\pi_0\colon\langle w,p\rangle\mapsto w$, and $A\subseteq A_f$ is tracked by $w\mapsto\langle w,f(w)\rangle$.
\end{proof}

Then we will quickly realize that there are numerous natural examples of non-regular subobjects.

\begin{example}[${\sf K}_2$ or ${\sf KV}$]\label{exa:subobject-fin-def}
In ${\sf K}_2$ and ${\sf KV}$, the space $\om^\om$ is represented by the identity map as in Example \ref{exa:code-itself}.
Then a subobject ${\sf Fin}$ of $\om^\om$ is defined as follows:
\begin{itemize}
\item The underlying set is $|{\sf Fin}|=\{x\in\om^\om:(\exists n)(\forall m\geq n)\;x(m)=0\}$.
\item A name of $x\in{\sf Fin}$ is a pair $\langle n,x\rangle$ of $x\in\om^\om$ and its witness $n$; that is, $x(m)=0$ for any $m\geq n$.
\end{itemize}

Note that the inclusion map ${\sf Fin}\mono\om^\om$ is a morphism, tracked by the projection $\pi_1\colon\langle n,x\rangle\mapsto x$.
\end{example}

\begin{prop}\label{prop:subobject-fin-def-nonreg-basic}
${\sf Fin}$ is a non-regular subobject of $\om^\om$.
\end{prop}

\begin{proof}
Suppose that ${\sf Fin}$ is regular.
Then, there exists a partial realizable function $F$ which, given $x\in|{\sf Fin}|$, returns an ${\sf Fin}$-name of $x$, say $F(x)=\langle n,x\rangle$.
By the continuity of $F$, the witness $n$ is determined after reading a finite initial segment $x\upto s$ of $x$.
Put $t=\max\{n,s\}$, and consider $y=(x\upto t)\fr 1\fr 0^\infty$.
Then $y\in |{\sf Fin}|$, so $F(y)$ returns a name of $y\in{\sf Fin}$; that is, $F(y)$ is of the form $\langle m,y\rangle$.
As $y$ extends $x\upto s$, the first value of $F(y)$ must be $n$; hence $F(y)=\langle n,y\rangle$.
However, we have $t\geq n$ and $y(t)=1$, which means that $\langle n,y\rangle$ is not an ${\sf Fin}$-name of $y$.
\end{proof}

One of the most typical ways to obtain a set is to describe a formula $\varphi$ to define a subset $\{x\in X:\varphi(x)\}$ of $X$.
Then, it is sometimes desirable to keep information behind the construction of the subset, for example, a witness of an existential quantification within $\varphi$.
In such a case, non-regular subobjects can appear, as described above.
Our goal is to classify such ``subsets with witnesses''.

\begin{remark}
Let us give some more background on Definition \ref{def:subobject-regular}.
In the case of sets, one can identify an injection $m\colon S\mono X$ with its image, which is a subset of $X$.
Then the inclusion relation between subsets of $X$ can be characterized using injections as follows.
A monomorphism $i\colon A\mono X$ is included in $j\colon B\mono X$ if there exists a morphism $k\colon A\to B$ such that $i=j\circ k$.
If $i$ is included in $j$ and vice versa, we write $i\equiv j$.
Formally, a subobject is the $\equiv$-equivalence class of a mono.
In the category of represented spaces, any $\equiv$-equivalence class of a mono contains an inclusion map, which leads us to Definition \ref{def:subobject-regular}.
Similarly, a regular subobject is the $\equiv$-equivalence class of a regular mono.
\end{remark}

\begin{remark}
One may also call a regular subobject of $X$ a {\em $\neg\neg$-closed} subobject of $X$.
For $j\colon\mathcal{P}({\tt Code})\to\mathcal{P}({\tt Code})$, the {\em $j$-closure} $A^j$ of a subobject $A\mono X$ is defined as follows:
\begin{align*}
|A^j|=|A|, & & E_{A^j}(x)=j(E_A(x))\land E_X(x),
\end{align*}
where recall that $E_A(x)$ is the set of all names of $x\in A$.
A subobject $A\mono X$ is {\em $j$-closed} if the $j$-closure $A^j\mono X$ is equivalent to $A\mono X$.

Then consider $\neg\neg\colon\mathcal{P}({\tt Code})\to\mathcal{P}({\tt Code})$ defined by $\neg\neg U=\mathcal{P}({\tt Code})$ if $U\not=\emptyset$; otherwise $\neg\neg U=\emptyset$.
One can easily see that a subobject is $\neg\neg$-closed iff it is regular.
To comment on the background of this notion, it is the closure by the universal closure operator obtained from the double negation topology.
\end{remark}

\begin{remark}
Example \ref{exa:subobject-fin-def} gives a direct definition of the function space $\om^\om$.
Alternatively, one may introduce $\om^\om$ as the exponential object ${\sf Nat}^{\sf Nat}$.
The notion of $\om^\om$ then makes sense even in ${\sf K}_1$ since the category of ${\sf K}_1$-represented spaces is cartesian closed, and in this case, ${\sf Nat}^{\sf Nat}$ is the space of all total computable functions.
Then one can formulate the definition of ${\sf Fin}$ within the system ${\sf K_1}$.
In fact, Proposition \ref{prop:subobject-fin-def-nonreg-basic} holds in ${\sf K}_1$ as well.
The details of this argument will be given later.
\end{remark}

\subsection{Lattice of subobjects}

Next, let us go a little further into the structure of the inclusion relation among subobjects.
Let us denote by ${\rm Sub}(X)$ the set of all subobjects of $X$.
One can easily check the following:

\begin{obs}
$({\rm Sub}(X),\subseteq)$ is a poset.
\end{obs}

In fact, in the category of represented spaces, one can see that a subobject poset is always a Heyting algebra.

\begin{prop}\label{prop:suboblect-poset-lattice}
The poset $({\rm Sub}(X),\subseteq)$ of subobjects of a represented space $X$ forms a Heyting algebra.
\end{prop}

This is a consequence of the fact that the category of represented spaces is a Heyting category, but it is important to give an explicit description of what the lattice and Heyting operations actually are.
Of course, as is well known, they are given in a form corresponding to the realizability interpretation.

\begin{definition}
Let $X,Y$ be subobjects of a represented space $Z$.
Then their {\em witnessed union} $X\wunion Y$ is defined as follows:
\begin{itemize}
\item The underlying set is $|X\wunion Y|=|X|\cup |Y|$.
\item $\langle i,p\rangle$ is a name of $x\in X\wunion Y$ iff, if $i=0$ then $p$ is a name of $x\in X$ else $p$ is a name of $x\in Y$.
\end{itemize}
\end{definition}

\begin{definition}
Let $X,Y$ be subobjects of a represented space $Z$.
Then their {\em witnessed intersection} $X\nplus Y$ is defined as follows:
\begin{itemize}
\item The underlying set is $|X\nplus Y|=|X|\cap |Y|$.
\item $\langle p,q\rangle$ is a name of $x\in X\nplus Y$ iff $p$ is a name of $x\in X$ and $q$ is a name of $x\in Y$.
\end{itemize}
\end{definition}


\begin{definition}
Let $X,Y$ be subobjects of a represented space $Z$.
Then the {\em implication} $X\arr Y$ is defined as follows:
\begin{itemize}
\item $\langle p,q\rangle$ is a name of $x\in X\arr Y$ iff $p$ is a name of $x\in Z$, and if $a$ is a name of $x\in X$ then $q\ast a$ is a name of $x\in Y$.
\item The underlying set is the set of all $x\in |Z|$ having $(X\arr Y)$-names.
\end{itemize}
\end{definition}

\begin{proof}[Proof of Proposition \ref{prop:suboblect-poset-lattice}]
We show that the witnessed union $\uplus$ and the intersection $\nplus$ give the join and the meet in any subobject poset.
Let $X,Y$ be subobjects of $Z$.
Then the inclusion maps $X\embed X\uplus Y$ and $Y\embed X\uplus Y$ are tracked by $a\mapsto\langle 0,a\rangle$ and $b\mapsto\langle 1,b\rangle$, respectively.
Now, let $S\mono Z$ be such that $X,Y\subseteq S$.
Then the inclusion maps $X\embed S$ and $Y\embed S$ are tracked by some $u$ and $v$, respectively.
Then consider the process that, given a name $\langle i,p\rangle$ of $x\in X\uplus Y$, if $i=0$ then returns $u(p)$ else $v(p)$.
Note that if $i=0$ then $p\vdash_X x$, so $u(p)\vdash_S x$.
Similarly, if $i\not=0$ then $p\vdash_Y x$, so $v(p)\vdash_Sx$.
Thus, in any case, the above process yields a name of $x\in S$.
This means that $X\uplus Y$ is included in $S$.
Hence, $X\uplus Y$ is the join of $X$ and $Y$ in the poset of subobjects of $Z$.

Next, the inclusion maps $X\nplus Y\embed X$ and $X\nplus Y\embed Y$ are tracked by projections $\pi_0$ and $\pi_1$, respectively.
Now, let $S\mono Z$ be such that $S\subseteq X,Y$.
Then the inclusion maps $S\embed X$ and $S\embed Y$ are tracked by some $u$ and $v$, respectively.
Then $p\mapsto\langle u(p),v(p)\rangle$ tracks the inclusion map $S\embed X\nplus Y$.
To see this, if $p\vdash_S x$ then $u(p)\vdash_Xx$ and $v(p)\vdash_Yx$, so $\langle u(p),v(p)\rangle\vdash_{X\nplus Y}x$.
This means that $S$ is included in $X\nplus Y$.
Hence, $X\nplus Y$ is the meet of $X$ and $Y$ in the poset of subobjects of $Z$.


To see that $\arr$ is the Heyting operation, for subobjects $A,B,C\mono Z$, first assume $A\subseteq B\arr C$, which is realized by $u$.
If $\pair{p,q}$ is a name of $x\in A\nplus B$ then, since $p$ is an $A$-name of $x$, $u(p)$ is a $(B\arr C)$-name of $x$ by our assumption.
If $u(p)$ is of the form $\pair{u_0(p),u_1(p)}$, as $q$ is a $B$-name of $x$, $u_1(p)\ast q$ is a $C$-name of $x$.
This shows $A\nplus B\subseteq C$.
Conversely, assume $A\nplus B\subseteq C$, which is realized by $u$.
Assume that $p$ is a name of $x\in A$.
As $A$ is a subobject of $Z$, the inclusion map $A\mono Z$ is tracked by some $i$, so $i(p)$ is a $Z$-name of $x$.
If $x\not\in B$ then anything is a $(B\arr C)$-name of $x$.
If $q$ is a name of $x\in B$ then $\langle p,q\rangle$ is a $(A\nplus B)$-name of $x$, so $u(p,q)$ is a $C$-name of $x$.
Hence, $\pair{i(p),\lambda q.u(p,q)}$ is a $(B\arr C)$-name of $x$.
Note that $\lambda q.u(p,q)$ is always defined, and thus $\lambda p.\pair{i(p),\lambda q.u(p,q)}$ tracks the inclusion $A\subseteq B\arr C$.
\end{proof}

\subsection{Quantifier}

One of our objectives is to analyze arithmetical and Borel hierarchies, especially in the latter case, it is useful to have the notion of countable union and countable intersection.
If a subobject lattice were complete, we could automatically obtain infinitary operations.
Of course, a lattice of regular subobjects in the category of represented spaces is always a complete Boolean algebra (since regular subobjects are merely subsets); however, the completeness of a subobject lattice strongly depends on what is chosen as a coding system.

%


\begin{prop}[${\sf K}_1$]
The poset $({\rm Sub}(X),\subseteq)$ of subobjects of a represented space $X$ is not necessarily a $\sigma$-complete lattice.
Indeed, $({\rm Sub}(\om),\subseteq)$ has neither $\sigma$-join nor $\sigma$-meet.
\end{prop}

\begin{proof}
For the non-existence of $\sigma$-join, for each $n\in\om$, think of the singleton $\{n\}$ as a subspace of $\om$.
Let $I\mono\om$ be such that $\{n\}\subseteq I$ in ${\rm Sub}(\om)$ for any $n\in\om$.
Then $|I|=\om$.
Let $E_I(n)$ be the set of all $I$-names of $n\in \om$.
Since $E_I(n)\not=\emptyset$ for all $n\in\om$, take the least element $s(n)\in E_I(n)$.
Construct a new subobject $J\mono\om$ as follows:
The underlying set is $|J|=\om$, and the only $J$-name of $n\in\om$ is $s'(n)$, where $s'$ be the Turing jump of $s$.
Clearly, $\{n\}\subseteq J$ holds in ${\rm Sub}(\om)$ for any $n\in\om$.
If $I\subseteq J$ in ${\rm Sub}(\om)$ then some computable function $f$ tracks the inclusion $I\embed J$.
Since $s(n)$ is an $I$-name of $n$, $f(s(n))$ must be a $J$-name of $n$, which means $f(s(n))=s'(n)$.
This implies that the jump $s'$ is Turing reducible to $s$, a contradiction.
Hence, $I\not\subseteq J$, which shows that any $I$ cannot be a $\sigma$-join of the singletons $\{n\}$.

For the non-existence of $\sigma$-meet, consider a sequence $g_0<_Tg_1<_T\dots$, where $\leq_T$ is Turing reducibility.
For each $i\in\om$, consider a subobject $A_i\mono\om$ such that its underlying set is $|A_i|=\om$ and $g_i(n)$ is a unique name of $n\in A_i$.
Assume that $B\mono\om$ is a subobject such that $B\subseteq A_i$ for any $i\in\om$.
We construct a subobject $C\mono\om$ such that $C\subseteq A_i$ for any $i\in\om$ but $C\not\subseteq B$.
Put $|C|=\om$.
We inductively define an increasing sequence $(n_e)_{e\in\om}$ with $n_0=0$.
Assume that $n_e$ has already been defined.
If the $e$th partial computable function $\varphi_e$ is not total then put $n_{e+1}=n_e+1$.
If $\varphi_e$ is total, we claim that there exists $n\geq n_e$ such that $\varphi_e(\langle g_i(n)\rangle_{i\leq e})$ is not a name of $n\in B$.
This is because, since $B\subseteq A_{e+1}$, if one knows a name of $n\in B$ one can compute $g_{e+1}(n)$.
Thus, if the claim fails
then we get $g_{e+1}\leq_T\bigoplus_{i\leq e}g_i\equiv_T g_e$, a contradiction.
Define $n_{e+1}=n+1$ where $n$ is a number in the above claim.
Then declare that, if $n_e\leq n<n_{e+1}$, then $\langle g_i(n)\rangle_{i\leq e}$ is a unique name of $n\in C$.
For almost all $n$, the $i$th coordinate of a name of $n\in C$ yields a name of $A_i$, so we have $C\subseteq A_i$.
Moreover, the above construction ensures that $\varphi_e$ cannot be a tracker of the inclusion $C\subseteq B$ for any $e$, which means that $C\not\subseteq B$.
Consequently, any $B$ cannot be a $\sigma$-meet of $(A_i)_{i\in\om}$.
\end{proof}

Of course, this only states that there are no external countable operations, and there is no problem if one uses internal countable operations or quantifiers.
In fact, the arithmetical hierarchy is usually defined as the hierarchy of number quantifiers.

It is known that the existential quantifier and the universal quantifier are characterized by the following adjoint rules:
\begin{align*}
A(x,y)\vdash B(y)&\iff \exists xA(x,y)\vdash B(y)\\
B(y)\vdash A(x,y)&\iff B(y)\vdash\forall xA(x,y)
\end{align*}

In categorical logic, the existential quantification and the universal quantification are introduced as follows:
If $A$ is a subobject of $X\times Y$, then $\exists^XA$ and $\forall^XA$ are subobjects of $Y$ such that for any subobject $B\mono Y$ the following holds in the subobject posets:
\begin{align*}
A\subseteq X\times B&\iff \exists^XA\subseteq B\\
X\times B\subseteq A&\iff B\subseteq\forall^XA
\end{align*}

As already mentioned, the category of represented spaces is a Heyting category, so it has interpretations of the existential quantification and the universal quantification.
To develop our theory, it is useful to have explicit descriptions of these notions.

\begin{definition}
Let $X$ be a subobject of a represented space $I\times Z$.
Then its {\em witnessed projection} $\exists^IX\mono Z$ is defined as follows:
\begin{itemize}
\item The underlying set is $|\exists^IX|=\{z\in Z:(\exists i\in I)\;(i,z)\in X\}$.
\item $\langle p,q\rangle$ is a name of $z\in \exists^IX$ iff $p$ is a name of some $i\in I$ and $q$ is a name of $z\in Z$ such that $(i,z)\in X$.
\end{itemize}
\end{definition}

From a practical standpoint, we also want to use the notation $\exists^ZX$ for the projection of $X\mono I\times Z$ into $I$.
Then, however, the notation $\exists^IX$ for $X\mono I\times I$ is ambiguous as to what it is quantifying.
In order to describe it unambiguously, we need to specify a projection, but it is often cumbersome and unintuitive.
In some cases, rather, a collection $(X_i)_{i\in I}$ of subobjects is given first, and the existential quantification is defined as its (witnessed) union.
To be more explicit:
\begin{definition}
Let $I$ be a represented space, and for each $i\in I$ let $X_i$ be a subobject of a represented space $Z$.
Then their {\em witnessed union} $\biguplus_{i\in I}X_i$ is defined as follows:
\begin{itemize}
\item The underlying set is $|\biguplus_{i\in I}X_i|=\bigcup_{i\in I}|X_i|$.
\item $\langle u,p\rangle$ is a name of $x\in \biguplus_{i\in I}X_i$ iff, if $u$ is a name of $j\in I$ then $p$ is a name of $x\in X_j$.
\end{itemize}
\end{definition}

Note that for a subobject $X\mono I\times Z$, we have $\exists^IX=\biguplus_{i\in I}X_i$, where $X_i$ is the subspace of $I\times Z$ whose underlying set is $\{z\in Z:(i,z)\in X\}$.


\begin{definition}
Let $I$ be a represented space, and for each $i\in I$ let $X_i$ be a subobject of a represented space $Z$.
Then their {\em witnessed intersection} $\bignplus_{i\in I}X_i$ is defined as follows:
\begin{itemize}
\item $p$ is a name of $x\in \bignplus_{i\in I}X_i$ iff, if $z$ is a name of $i\in I$ then $p\ast z$ is a name of $x\in X_i$.
\item The underlying set is the set of all $x\in |Z|$ having $\bignplus_{i\in I}X_i$-names.
\end{itemize}
\end{definition}

These notions plays the roles of the existential quantification and the universal quantification, respectively; that is, for any subobjects $A\mono I\times Z$ and $B\mono Z$,
\begin{align*}
A\subseteq I\times B&\iff \biguplus_{i\in I}A_i\subseteq B\\
I\times B\subseteq A&\iff B\subseteq\bignplus_{i\in I}A_i
\end{align*}
where $A_i$ is the subobject of $Z$ such that $|A_i|=\{z\in |Z|:(i,z)\in |A|\}$ and a name of $z\in A_i$ is a name of $(i,z)\in A$.

\subsection{Reducibility}

We will see later, for example, that ${\sf Fin}$ is not $\Sigma^0_2$-complete.
Of course, in order to define the notion of completeness, we need a notion of reducibility, which requires some discussion.

\begin{definition}\label{def:usual-many-one}
For subsets $A,B\subseteq\om$, we say that $A$ is {\em many-one reducible to} $B$ (written $A\leq_mB$) if there exists a computable function $\varphi\colon\om\to\om$ such that for any $n\in\om$
\[n\in A\iff \varphi(n)\in B.\]

For a collection $\Gamma$ of subsets of $\om$, a set $A\subseteq\om$ is {\em $\Gamma$-hard} if $B\leq_mA$ for any $B\in\Gamma$.
If $A\in\Gamma$ also holds, then $A$ is called {\em $\Gamma$-complete}.
\end{definition}

\begin{definition}\label{def:Wadge-reducibility}
For subsets $A,B\subseteq\om^\om$, we say that $A$ is {\em Wadge reducible to} $B$ (written $A\leq_WB$) if there exists a continuous function $\varphi\colon\om^\om\to\om^\om$ such that for any $x\in\om^\om$
\[x\in A\iff \varphi(x)\in B.\]

For a collection $\Gamma$ of subsets of $\om^\om$, a set $A\subseteq\om$ is {\em $\Gamma$-hard} if $B\leq_WA$ for any $B\in\Gamma$.
If $A\in\Gamma$ also holds, then $A$ is called {\em $\Gamma$-complete}.
\end{definition}

These notions of reducibility are for subsets and can be easily extended to subspaces (or regular subobjects) of represented spaces, but it is less clear how to extend them to non-regular subobjects.
After a short time of gazing at the definitions, one may find that both definitions of reducibility can be rewritten as $A=\varphi^{-1}[B]$.
That is, this is just a pullback.
Using the notion of pullback, it is natural to extend the notion of many-one/Wadge reducibility as follows:

\begin{definition}\label{def:pullback-many-one-reducibility}
Let $X,Y$ be objects in a category $\mathcal{C}$ having pullbacks.
A mono $A\monoedge{i} X$ is {\em many-one reducible to} $B\monoedge{j}Y$ if $A\monoedge{i}X$ is a pullback of $B\monoedge{j}Y$ along some morphism $\varphi\colon X\to Y$.
\[
\xymatrix{
A \ar[rr] \ar@{ >->}[d]  \pushoutcorner & & B\ar@{ >->}[d] \\
X \ar[rr]_\varphi & & Y 
}
\]
\end{definition}

Recall an explicit description of pullback in the category of represented spaces:

\begin{definition}\label{def:pullback-in-represented-spaces}
For represented spaces $X$ and $Y$, for a subobject $B\mono Y$ and a morphism $\varphi\colon X\to Y$, the {\em pullback} $\varphi^\ast B$ of $B$ along $\varphi$ is defined as follows:
\begin{itemize}
\item The underlying set is $|\varphi^\ast B|=\varphi^{-1}[|B|]=\{x\in|X|:\varphi(x)\in|B|\}$.
\item A name of $x\in\varphi^\ast B$ is a pair $\langle p,q\rangle$ of a name $p$ of $x\in X$ and a name $q$ of $\varphi(x)\in B$.
\end{itemize}

If $B\mono Y$ is regular, then so is $\varphi^\ast B\mono X$, and the information of $q$ in the name $\varphi^\ast B$ is unnecessary.
\end{definition}

The definition of many-one reducibility in the category of represented spaces can be explicitly written as follows:

\begin{definition}\label{def:represented-many-one-reducibility}
Let $X$ and $Y$ be represented spaces.
A subobject $A\mono X$ is {\em many-one reducible to} $B\mono Y$ if there exist a morphism $\varphi\colon X\to Y$ and partial realizable functions $r_-,r_+\pcolon{\tt Code}\to{\tt Code}$ such that the following hold:
\begin{enumerate}
\item For any $x\in |X|$, $x\in |A|$ if and only if $\varphi(x)\in |B|$.
\item If $p$ is an $A$-name of $x\in|A|$ then $r_-(p)$ is a $B$-name of $\varphi(x)\in |B|$.
\item If $p$ is an $X$-name of $x\in|A|$ and $q$ is a $B$-name of $\varphi(x)\in |B|$ then $r_+(p,q)$ is an $A$-name of $x\in |A|$.
\end{enumerate}
\end{definition}

In this case, we write $A\leq_{\sf m}B$.
One can think of this as the realizability interpretation of many-one reducibility; that is, $r_-$ is a realizer for ``$x\in A\implies \varphi(x)\in B$'', and $r_+$ is a realizer for ``$\varphi(x)\in B\implies x\in A$''.

\begin{obs}
Definition \ref{def:represented-many-one-reducibility} is equivalent to Definition \ref{def:pullback-many-one-reducibility} in the category of represented spaces.
\end{obs}

\begin{proof}
Definition \ref{def:represented-many-one-reducibility} states that the subobject $A\mono X$ is equivalent to the pullback $\varphi^\ast B\mono X$ in Definition \ref{def:pullback-in-represented-spaces}.
Indeed, if $u$ is a tracker of $\varphi$, then $p\mapsto \pair{ u(p),r_-(p)}$ tracks $A\subseteq \varphi^\ast B$, and $r_+$ tracks $\varphi^\ast B\subseteq A$.
Thus, $A\mono X$ is also a pullback of $B\mono Y$ along $\varphi$.
Conversely, if a subobject $A\mono X$ is a pullback of $B\mono Y$ along $\varphi$, then one can easily see that $A$ is equivalent to $\varphi^\ast B$.
\end{proof}

\begin{example}
Let ${\bf Rep}({\sf K}_1)$ be the category of represented spaces over Kleene's first algebra ${\sf K}_1$.
Then many-one reducibility for regular subobjects of $\om$ in ${\bf Rep}(K_1)$ is exactly the usual {\em many-one reducibility}.
\end{example}

\begin{example}
Let ${\bf Rep}({\sf K}_2)$ be the category of represented spaces over Kleene's second algebra ${\sf K}_2$.
Then many-one reducibility for regular subobjects of $\om^\om$ in ${\bf Rep}(K_2)$ is exactly the {\em Wadge reducibility}.
\end{example}

\begin{example}\label{exa:computable-Levin-reducibility-b}
Many-one reducibility for subobjects of $\om$ in ${\bf Rep}(K_1)$ can be thought of as {\em computable Levin reducibility} (cf.~Definition \ref{def:Levin-reducibility}).

To be more explicit, a witnessed subset $A\mono X$ is many-one reducible to a witnessed subset $B\mono Y$ if and only if there exist a morphism $\varphi\colon X\to Y$ and partial realizable functions $r_-,r_+$ such that for any name $p$ of $x\in X$ and $v,w\in{\sf Code}$ the following holds:
\begin{enumerate}
\item $x\in |A|$ if and only if $\varphi(x)\in |B|$.
\item If $v$ is a witness for $x\in|A|$ then $r_-(v,p)$ is a witness for $\varphi(x)\in |B|$.
\item If $w$ is a witness for $\varphi(x)\in|B|$ then $r_+(w,p)$ is a witness for $x\in |A|$.
\end{enumerate}
\end{example}

Next, let us discuss an alternative to the definition of many-one reducibility.
Considering the ``meaning'' of many-one reducibility, it might be a bit questionable whether Definition \ref{def:represented-many-one-reducibility} is appropriate.
The ``meaning'' of many-one reducibility $A\leq_mB$ is that it is sufficient to know information about $B$ in order to compute information about $A$.
That is, the computationally correct understanding of $A\leq_mB$ might not be ``$x\in A\iff\varphi(x)\in B$,'' but
\[\varphi(x)\in B\implies x\in A;\mbox{ and }\varphi(x)\not\in B\implies x\not\in A.\]

The conditions (1) and (3) of Definition \ref{def:represented-many-one-reducibility} reflect that meaning, but the condition (2) is the opposite.
Therefore, from a computational point of view, it would be correct to remove the condition (2).
In fact, for reducibility for function problems in computational complexity theory (e.g., the definition of {\rm FNP}-completeness), this form of reducibility is often considered instead of Levin reducibility.

\begin{definition}
We say that a subobject $A\mono X$ is {\em demi-many-one reducible to} $B\mono Y$ (written $A\leq_{\sf m}'B$) if the conditions (1) and (3) in Definition \ref{def:represented-many-one-reducibility} hold.
\end{definition}

\begin{obs}
$A\leq_{\sf m}B\implies A\leq_{\sf m}'B$.
\end{obs}

Of course, removing the condition (2) would be unnatural in the categorical setting.
Recall that the $\neg\neg$-closure $A^{\neg\neg}\mono X$ of a subobject $A\mono X$ as the unique regular subobject with $|A^{\neg\neg}|=|A|$.
In other words, it is the $\subseteq$-least regular subobject including $A$.

\begin{obs}
$A\leq_{\sf m}'B$ iff there exists a morphism $\varphi$ such that $\varphi^\ast B\subseteq A\subseteq(\varphi^\ast B)^{\neg\neg}$.
\end{obs}

\begin{obs}
If $A,B\mono X$ are regular then $A\leq_{\sf m}B$ iff $A\leq_{\sf m}'B$.
\end{obs}

\begin{remark}
One may notice that demi-many-one reducibility closely resembles Weihrauch reducibility \cite{pauly-handbook}, which has been studied in depth in computable analysis.
In fact, demi-many-one reducibility corresponds to Weihrauch reducibility for ``hardest totalizations''.
Here, the hardest totalization $F^{ht}$ of $F\pcolon X\tto Y$ is the totalization of $F$ such that if $x\not\in{\rm dom}(F)$ then $F^{ht}(x)=\emptyset$.
\end{remark}

\subsection{Structure}

Let us analyze the structure of many-one reducibility.

\begin{obs}
$\leq_{\sf m}$ and $\leq_{\sf m}'$ form preorders.
\end{obs}

\begin{proof}
The proof for the many-one case follows from a general argument, but we give an explicit description applicable to the demi-many-one case.
Reflexivity is trivial.
For transitivity, for subobjects $A,B,C$ of $X,Y,Z$, assume $A\leq_{\sf m}B$ via $\varphi,r_-,r_+$ and $B\leq_{\sf m}C$ via $\psi,s_-,s_+$.
Only an outer reduction is nontrivial.
Let $p$ be a tracker of $\varphi$.
Given a name $a$ of $x\in X$ we know that $p\ast a$ is a name of $\varphi(x)\in Y$.
Thus, given a name $c$ of $\psi(\varphi(x))\in C$, $s^+(p\ast a,c)$ is a name of $\varphi(x)\in B$, and $r^+(a,s^+(p\ast a,c))$ is a name of $x\in A$.
This process gives an outer reduction for $A\leq_{\sf m}C$.
This actually shows that $\leq_{\sf m}'$ is a preorder.
\end{proof}

The poset reflections of $\leq_{\sf m}$ and $\leq_{\sf m}'$ are called the many-one degrees and the demi-many-one degrees, respectively.
It is easy to show the following property similar to many order degree structures.

\begin{prop}
The (demi-)many-one degrees on represented spaces form an upper semilattice.
\end{prop}

\begin{proof}
Again, the proof for the many-one case follows from a general argument, but we give an explicit description applicable to the demi-many-one case.
We construct a join of subobjects $A\mono X$ and $B\mono Y$.
The coproduct $X+Y$ of represented spaces $X$ and $Y$ are defined as follows:
\begin{itemize}
\item The underlying set is $|X+Y|=\{(0,x):x\in X\}\cup\{(1,y):y\in Y\}$.
\item $\pair{i,p}$ is a name of $(j,z)$ iff $i=j$ and if $i=0$ then $p$ is a name of $z\in X$ else $p$ is a name of $z\in Y$.
\end{itemize}

One can think of $A+B$ as a subobject of $X+Y$.
Clearly, $A,B\leq_{\sf m}A+B$.
Assume that $C$ is a subobject of $Z$ such that $A\leq_{\sf m}C$ via $\varphi,r_-,r_+$ and $B\leq_{\sf m}C$ via $\psi,s_-,s_+$.
Let $[\varphi,\psi]\colon X+Y\to Z$ be a morphism such that, for any $(i,x)\in X+Y$, if $i=0$ then $[\varphi,\psi](i,x)=\varphi(x)$ else $[\varphi,\psi](i,x)=\psi(x)$.
Given $\pair{i,u}\vdash_{A+B}(i,x)$, if $i=0$ then $r_-(u)\vdash_C\varphi(x)=[\varphi,\psi](i,x)$ else $s_-(u)\vdash_C\psi(x)=[\varphi,\psi](i,x)$.
Given $\pair{i,u}\vdash_{X+Y}(i,x)$ and $c\vdash_C[\varphi,\psi](i,x)$, if $i=0$ then $[\varphi,\psi](i,x)=\varphi(x)$, so $r_+(u,c)\vdash_Ax$; thus $\pair{0,r_+(u,c)}\vdash_{A+B}(i,x)$.
Similarly, if $i\not=0$ then $\pair{1,r_+(u,c)}\vdash_{A+B}(i,x)$.
This shows $A+B\leq_{\sf m}C$.
By a similar argument, one can show that $\leq_{\sf m}'$ also yields an upper semilattice.
\end{proof}

The many-one degree structure has the following good property that other degree structures do not have very often.

\begin{theorem}
The (demi-)many-one degrees on represented spaces form a distributive lattice.
\end{theorem}

\begin{proof}
First consider $\leq_{\sf m}'$.
Given subobjects $A\mono X$ and $B\mono Y$, we define the space $A\odot B$ of common information of $A$ and $B$ as follows:
First consider the subspace $C(A,B)$ of ${\tt Code}$ such that $\pair{p,q,n}\in |C(A,B)|$ iff 
\[(\exists x,y)\;[({p\ast n\downarrow}\vdash_Xx)\ \&\ ({q\ast n\downarrow}\vdash_Yy)\ \&\ (x\in |A|\iff y\in |B|)].\]

Then the equivalence relation $\sim$ on $C(A,B)$ is defined as follows:
\[\pair{p,q,n}\sim\pair{p,q,m}\iff \delta_X(p\ast n)=\delta_X(p\ast m)\ \&\ \delta_Y(q\ast n)=\delta_Y(q\ast m).\]

Then let $A\odot B$ be the quotient space $C(A,B)/{\sim}$; that is, an $(A\odot B)$-name of a $\sim$-equivalence class $[p,q,n]$ is of the form $\pair{p,q,m}$ for some $\pair{p,q,m}\in[p,q,n]$.
Then we define a subobject $A\Cap B\mono A\odot B$ as follows:
\begin{itemize}
\item $[p,q,n] \in |A\Cap B|$ iff $[p,q,n]\in |A\odot B|$ and $\delta_X(p\ast n)\in |A|$.
\item $\langle i,u,v\rangle$ is a name of $[p,q,n]\in A\Cap B$ iff $u$ is a name of $[p,q,n]\in A\odot B$ and if $i=0$ then $v$ is a name of $\delta_X(p\ast n)\in A$ else $v$ is a name of $\delta_Y(q\ast n)\in B$.
\end{itemize}

To see that $A\Cap B\leq_{\sf m}' A$, consider $\varphi\colon[p,q,n]\mapsto\delta_X(p\ast n)$, which is well-defined.
If $[p,q,n]\in|A\odot B|$ then ${\varphi([p,q,n])\downarrow}\in X$ since $\pair{p,q,n}\in C(A,B)$.
Thus, $[p,q,n]\in |A\Cap B|$ iff $\varphi([p,q,n])\in |A|$.
Moreover, if $\pair{p,q,m}$ is a name of $[p,q,m]\in A\odot B$ and $v$ is a name of $\varphi([p,q,n])\in A$ then $\langle 0,p,q,m,v\rangle$ is a name of $[p,q,n]\in A\Cap B$.

Similarly, to see that $A\Cap B\leq_{\sf m}' B$, consider $\varphi\colon[p,q,n]\mapsto\delta_Y(q\ast n)$, which is well-defined.
If $[p,q,n]\in|A\odot B|$ then ${\varphi([p,q,n])\downarrow}\in Y$.
By definition, $[p,q,n]\in |A\Cap B|$ iff $\delta_X(p\ast n)\in |A|$, iff $\varphi([p,q,n])=\delta_X(q\ast n)\in |B|$ since $\langle p,q,n\rangle\in C(A,B)$.
Moreover, if $\pair{p,q,m}$ is a name of $[p,q,m]\in A\odot B$ and $v$ is a name of $\varphi([p,q,n])\in B$ then $\langle 1,p,q,m,v\rangle$ is a name of $[p,q,n]\in A\Cap B$.

For a subobject $C\mono Z$, if $C\leq_{\sf m}'A,B$ then there are $\varphi,\psi$ such that, for any $x\in|Z|$, $x\in |C|$ iff $\varphi(x)\in|A|$ and $\psi(x)\in |B|$.
We also have $r,s$ such that if $u$ is a name of $\varphi(x)\in A$ then $r\ast u$ is a name of $x\in C$, and if $v$ is a name of $\psi(x)\in B$ then $s\ast v$ is a name of $x\in C$.

To show $C\leq_{\sf m}'A\Cap B$, let $p$ and $q$ be trackers of $\varphi$ and $\psi$, respectively.
Then if $k$ is a name of some $x\in Z$, then $k\vdash_Z x\in |C|$ iff ${p\ast k\downarrow}\vdash_X\varphi(x)\in |A|$ iff ${q\ast k\downarrow}\vdash_Y\psi(x)\in |B|$.
Hence, $[p,q,k]\in|A\odot B|$, and moreover, $k\vdash_Z x\in |C|$ iff $[p,q,k]\in |A\Cap B|$.
Moreover, if $k$ and $\ell$ are names of $x\in Z$, then $\delta_X(p\ast k)=\delta_X(p\ast \ell)=\varphi(x)$ and $\delta_Y(q\ast k)=\delta_Y(q\ast \ell)=\psi(x)$; hence $[p,q,k]=[p,q,\ell]$.
Therefore, $k\mapsto \langle p,q,k\rangle$ tracks some morphism $\theta\colon Z\to A\odot B$ such that $x\in|C|$ iff $[p,q,k]\in |A\Cap B|$.
Given a name $\langle i,u,v\rangle$ of $[p,q,k]\in A\Cap B$, if $i=0$ then $v$ is a name of $\delta_X(p\ast k)=\varphi(x)\in A$, so $r\ast v$ is a name of $x\in C$.
If $i\not=0$ then $v$ is a name of $\delta_Y(q\ast k)=\psi(x)\in B$, so $s\ast v$ is a name of $x\in C$.
This completes the proof.

The proof for $\leq_{\sf m}$ is almost the same, but the construction is a little more complicated.
Consider the subspace $D(A,B)$ of ${\tt Code}$ such that $\pair{p,q,a,b,c,d,n}\in |D(A,B)|$ iff $\pair{p,q,n}\in|C(A,B)|$ and the following hold:
\begin{itemize}
\item If $v$ is a name of $\delta_X(p\ast n)\in A$ then $a\ast(c\ast v)$ is a name of $\delta_X(p\ast n)\in A$ and $b\ast(c\ast v)$ is a name of $\delta_Y(q\ast n)\in B$.
\item If $v$ is a name of $\delta_Y(q\ast n)\in B$ then $a\ast(d\ast v)$ is a name of $\delta_X(p\ast n)\in A$ and $b\ast(d\ast v)$ is a name of $\delta_Y(q\ast n)\in B$. 
\end{itemize}

The construction after this is exactly the same as in the previous one, except for the additional information $\alpha=(a,b,c,d)$.

To see that $A\Cap B\leq_{\sf m} A$, consider $\varphi\colon[p,q,\alpha,n]\mapsto\delta_X(p\ast n)$, which is well-defined.
If $[p,q,\alpha,n]\in|A\odot B|$ then ${\varphi([p,q,\alpha,n])\downarrow}\in X$ since $\pair{p,q,n}\in C(A,B)$.
Thus, $[p,q,\alpha,n]\in |A\Cap B|$ iff $\varphi([p,q,\alpha,n])\in |A|$.
Moreover, if $\pair{p,q,\alpha,m}$ is a name of $[p,q,\alpha,m]\in A\odot B$ and $v$ is a name of $\varphi([p,q,\alpha,n])\in A$ then $\langle 0,p,q,\alpha,m,v\rangle$ is a name of $[p,q,\alpha,n]\in A\Cap B$.
Conversely, assume that $\langle 0,p,q,\alpha,m,v\rangle$ is a name of $[p,q,\alpha,n]\in A\Cap B$.
If $i=0$ then $v$ is a name of $\delta_X(p\ast n)\in A$.
If $i=1$ then $v$ is a name of $\delta_Y(q\ast n)\in B$, so $a\ast(d\ast v)$ is a name of $\delta_X(p\ast n)\in A$.
By the same argument, one can also show $A\Cap B\leq_{\sf m} B$.

For a subobject $C\mono Z$, if $C\leq_{\sf m}A,B$ then there are $\varphi,\psi$ such that, for any $x\in|Z|$, $x\in |C|$ iff $\varphi(x)\in|A|$ and $\psi(x)\in |B|$.
Let $p$ and $q$ be trackers of $\varphi$ and $\psi$, respectively.
Then if $k$ is a name of some $x\in Z$, then $k\vdash_Z x\in |C|$ iff ${p\ast k\downarrow}\vdash_X\varphi(x)\in |A|$ iff ${q\ast k\downarrow}\vdash_Y\psi(x)\in |B|$.
We also have $a,b,c,d$ such that $v\vdash_A\varphi(x)$ implies $c\ast v\vdash_Cx$; $v\vdash_B\psi(x)$ implies $d\ast v\vdash_Cx$; and $v\vdash_Cx$ implies $a\ast v\vdash_A\varphi(x)$ and $b\ast v\vdash_B\psi(x)$.
In particular, $v\vdash_A\varphi(x)$ implies $a\ast (c\ast v)\vdash_A\varphi(x)$ and $b\ast (c\ast v)\vdash_B\psi(x)$, and $v\vdash_B\psi(x)$ implies $a\ast (d\ast v)\vdash_A\varphi(x)$ and $b\ast (d\ast v)\vdash_B\psi(x)$.
As $\varphi(x)=\delta_X(p\ast k)$ and $\psi(x)=\delta_Y(q\ast k)$, this implies $\langle p,q,\alpha,n\rangle\in D(A,B)$.
The rest of the proof of $C\leq_{\sf m}A\Cap B$ is the same as before.

To show distributivity, first consider $\leq_{\sf m}'$.
For subobjects $A,B,C$ of $X,Y,Z$, it suffices to show $(A+B)\Cap (A+C)\leq_{\sf m}' A+(B\Cap C)$.
We construct a reduction $\varphi$ as follows:
Given $[p,q,n]\in|(A+B)\odot(A+C)|$, assume that $p\ast n$ and $q\ast n$ are of the forms $\langle i,u\rangle$ and $\langle j,v\rangle$, respectively.
If $i=0$ then define $\varphi([p,q,n])=\pair{0,\delta_X(u)}$ else if $j=0$ then $\varphi([p,q,n])=\pair{0,\delta_X(v)}$.
For instance, if $i=1$ and $j=0$ then $\delta_Y(u)\in |B|$ iff $\delta_X(v)\in |A|$; hence $[p,q,n]\in |(A+B)\Cap (A+C)|$ iff $\delta_X(v)\in |A|$ iff $\varphi([p,q,n])=\pair{0,\delta_X(v)}\in |A+(B\Cap C)|$.
If $i\not=0$ and $j\not=0$ then let $p_1$ and $q_1$ be such that $p_1\ast n=u$ and $q_1\ast n=v$ and then define $\varphi([p,q,n])=\langle 1,[p_1,q_1,n]\rangle$.
As $\delta_Y(u)\in |B|$ iff $\delta_Z(v)\in |C|$ iff $[p_1,q_1,n]\in |B\Cap C|$ iff $\varphi([p,q,n])\in |A+(B\Cap C)|$.
For an outer reduction, given a name $\pair{p,q,m}$ of $[p,q,n]\in(A+B)\odot(A+C)$ and a name $\langle k,w\rangle$ of $\varphi[p,q,n]\in A+(B\Cap C)$, if $i=0$ or $j=0$ then $k=0$ and $w$ is an $A$-name of $\delta_X(u)$ or $\delta_X(v)$, so one can compute a name of $[p,q,n]\in(A+B)\Cap(A+C)$.
If $i\not=0$ and $j\not=0$ then $w$ is of the form $\pair{\ell,p_1,q_1,m',z}$, where if $\ell=0$ then $z\vdash_B\delta_Y(u)$ else $z\vdash_C\delta_Z(v)$.
Thus, if $\ell=0$ then $\langle 1,z\rangle\vdash_{A+B}(1,\delta_Y(u))$ else $\langle 1,z\rangle\vdash_{A+C}(1,\delta_Z(v))$, so one can compute a name of $[p,q,n]\in(A+B)\Cap(A+C)$.

For $\leq_{\sf m}$, if $i\not=0$ and $j\not=0$ then transform $\alpha=(a,b,c,d)$ into $\alpha'=(a',b',c',d')$, where $a'=\lambda x.\pi_1(a\ast x)$, $b'=\lambda x.\pi_1(b\ast x)$, $c'=\lambda x.c(1,x)$ and $d'=\lambda x.d(1,x)$.
Then define $\varphi([p,q,\alpha,n])=\langle 1,[p_1,q_1,\alpha',n]\rangle$.
The rest of the discussion is the same as above.
\end{proof}

However, if one focuses only on subobjects of a single represented space, it is generally not a lattice.

\begin{prop}[${\sf K}_1$]\label{prop:many-one-non-meet}
Both $({\rm Sub}(\om),\leq_{\sf m})$ and $({\rm Sub}(\om),\leq_{\sf m}')$ do not have binary meet.
\end{prop}

Let us prepare a lemma to prove this.
A general argument shows that regular subobjects are $\leq_{\sf m}$-downward closed.

\begin{lemma}\label{lem:n-regular-downclosed1}
If a mono $i\colon A\mono X$ is many-one reducible to a regular mono $j\colon B\mono Y$, then $i\colon A\mono X$ is also a regular mono.
\end{lemma}

\begin{proof}
Regular monos are stable under pullback.
\end{proof}

This can also easily be shown using explicit descriptions.
In fact, explicit descriptions lead us to more than that.

\begin{lemma}\label{lem:n-regular-downclosed}
If a mono $i\colon A\mono X$ is demi-many-one reducible to a regular mono $j\colon B\mono Y$, then $i\colon A\mono X$ is also a regular mono.
\end{lemma}

\begin{proof}
Assume $A\leq_{\sf m}' B$ via $\varphi,r_+$.
Given an $X$-name $p$ of $x\in|A|$, using a tracker of $\varphi$, one can find an $X$-name of $\varphi(x)\in |B|$.
As $B$ is regular, we get a $B$-name $q$ of $\varphi(x)$.
Then $r_+(p,q)$ is an $A$-name of $x$.
Thus, $A$ is regular.
\end{proof}

\begin{proof}[Proof of Proposition \ref{prop:many-one-non-meet}]
For the non-existence of meet, it is known that every increasing sequence of many-one degrees has an exact pair \cite[Theorem VI.3.4]{OdiBook}; that is, for any sequence $A_0<_mA_1<_m\dots$ of subsets of $\om$ there exists $B,C\subseteq\om$ such that $A\leq_m B,C$ if and only if $A\leq_m A_i$ for some $i\in\om$, where $\leq_m$ is many-one reducibility in the classical sense (Definition \ref{def:usual-many-one}).
Now note that if $R,S$ are regular subobjects of $\om$ then $R\leq_{\sf m}S$ iff $R\leq_{\sf m}'S$ iff $|R|\leq_m|S|$.
For each $i\in\om$, let us think of $A_i,B,C$ as regular subobjects of $\om$.
We claim that $B$ and $C$ do not have meet.
If $A\leq_{\sf m}' B,C$ then, by Lemma \ref{lem:n-regular-downclosed}, $A$ is also regular.
Thus, the above property shows $A\leq_{\sf m}A_i$ for some $i\in\om$.
Then we get $A<_{\sf m}A_{i+1}<_{\sf m}B,C$, which means that $A$ is not a meet of $B$ and $C$.
\end{proof}

\section{Hierarchy}

\subsection{Sierpi\'nski dominance}

Our objective is to analyze the arithmetical/Borel hierarchy in the category of represented spaces.
For this, we first need to define the notion of $\Sigma^0_1$ and $\Pi^0_1$ subobjects.
Let us first explain these notions for subobjects of $\om^\om$ in ${\sf K}_2$ or ${\sf KV}$, followed by the general definitions.

\begin{definition}[${\sf K}_2$ or ${\sf KV}$]\label{def:K2KV-open-closed}
A subobject $A\mono\om^\om$ is {\em open} or $\Sigma^0_1$ if there exists a morphism $\varphi\colon\om^\om\times\om\to\om$ such that $A$ is equivalent to the following subspace $E_\varphi$ of $\om^\om$: 
\[E_\varphi=\{x\in\om^\om:(\exists n\in\om)\;\varphi(x,n)\not=0\}.\]

Similarly, a subobject $A\mono\om^\om$ is {\em closed} or $\Pi^0_1$ if there exists a morphism $\varphi\colon\om^\om\times\om\to\om$ such that $A$ is equivalent to the subspace of $\om^\om$ whose underlying set is $\om^\om\setminus E_\varphi$. 
\end{definition}

Via currying, a $\Sigma^0_1$ subobject is a subspace of the form $\{x\in\om^\om:\varphi(x)\not=0^\infty\}$, and a $\Pi^0_1$ subobject is a subspace of the form $\{x\in\om^\om:\varphi(x)=0^\infty\}$.
We would like to consider a similar notion for ${\sf K}_1$, but we need to give a formal definition of $\om^\om$ in ${\sf K}_1$.

\begin{definition}
By an abuse of notation, we use $\om^\om$ to denote the exponential object ${\sf Nat}^{\sf Nat}$.
To be more explicit:
\begin{itemize}
\item In ${\sf K}_2$ and ${\sf KV}$, the underlying set $|{\sf Nat}^{\sf Nat}|$ is the set of all functions on $\N$.
A name of $f\in{\sf Nat}^{\sf Nat}$ is $f$ itself (as in Examples \ref{exa:code-itself} and \ref{exa:subobject-fin-def}).
\item In ${\sf K}_1$, the underlying set $|{\sf Nat}^{\sf Nat}|$ is the set of all total computable functions on $\N$.
A name of $f\in{\sf Nat}^{\sf Nat}$ is any program $p\in{\tt Code}$ computing $f$.
\end{itemize}
\end{definition}

By adopting this definition, Definition \ref{def:K2KV-open-closed} makes sense for ${\sf K}_1$ as well.
Since most of the concrete examples in this article are subobjects of $\om^\om$, it is sufficient to understand the above as definitions of open and closed subobjects.
However, for the sake of uniform discussion, we also give general definitions.

\begin{definition}\label{def:open-closed-general}
A subobject $A\mono X$ is {\em open} or $\Sigma^0_1$ if there exists a partial realizable function $F$ such that $A$ is equivalent to the following subspace $E_F$ of $X$:
\begin{align*}
x\in |E_F|&\iff \mbox{if $p$ is an $X$-name of $x$ then $F(p)$ is an $\om^\om$-name of $0^\infty$},\\
x\not\in |E_F|&\iff \mbox{if $p$ is an $X$-name of $x$ then $F(p)$ is an $\om^\om$-name of some $\alpha\not=0^\infty$}.
\end{align*}

Similarly, a subobject $A\mono\om^\om$ is {\em closed} or $\Pi^0_1$ if there exists a partial realizable function such that $A$ is equivalent to the subspace of $X$ whose underlying set is $|X|\setminus |E_F|$. 
\end{definition}

Let us explain the general theory behind this definition.
Recall from Example \ref{exa:Sierpinski-space} that $\mathbb{S}$ is the represented Sierpi\'nski space.
Here, (an $\om^\om$-name of) the sequence $0^\infty$ is a name of $\top\in\mathbb{S}$ and any other sequence is a name of $\bot\in\mathbb{S}$.

%

It is well-known that an open set in a topological space can be identified with a continuous map to the (topological) Sierpi\'nski space $\mathbb{S}$, which consists of the open point $\top$ and the closed point $\bot$.
To be precise, an open subset $A$ of a topological space $X$ is exactly a set of the form $\varphi^{-1}\{\top\}$ for some continuous map $\varphi\colon X\to\mathbb{S}$.
This idea can be generalized as follows:

\begin{definition}
A subobject $A\mono X$ is {\em open} if it is a pullback of $\top\colon\mathbf{1}\mono\mathbb{S}$ along some morphism $\varphi\colon X\to\mathbb{S}$.
Similarly, a subobject $A\mono X$ is {\em closed} if it is a pullback of $\bot\colon\mathbf{1}\mono\mathbb{S}$ along some morphism $\varphi\colon X\to\mathbb{S}$.
\end{definition}

To be more explicit, the pullback $\varphi^\ast\top$ can be written as follows:
\begin{itemize}
\item The underlying set is $|\varphi^\ast\top|=\{x\in X:\varphi(x)=\top\}$.
\item A name of $x\in\varphi^\ast\top$ is a name of $x\in X$.
\end{itemize}

Note that fixing a name of $\top\in\mathbb{S}$ in advance makes it possible to omit the information on a name of $\top$ form the above description of $\varphi^\ast\top$.
By definition of the representation of $\mathbb{S}$, it is easy to verify that this definition is consistent with Definition \ref{def:open-closed-general}.
Also note, by definition, that $A\mono X$ is open iff $A\leq_{\sf m}(\top\colon\mathbf{1}\mono\mathbb{S})$, and $A\mono X$ is closed iff $A\leq_{\sf m}(\bot\colon\mathbf{1}\mono\mathbb{S})$.

\begin{obs}\label{obs:open-closed-is-regular}
An open subobject is regular.
Similarly, a closed subobject is regular.
\end{obs}

\begin{proof}
This is because an open subobject is a pullback of a regular subobject $\top\colon\mathbf{1}\mono\mathbb{S}$.
The same applies to a closed subobject.
See also the above explicit description of the pullback.
\end{proof}

\begin{example}~
\begin{enumerate}
\item In ${\sf K}_1$:
An open subobject of $\om$ is exactly a $\Sigma^0_1$ subset of $\om$, and a closed subobject of $\om$ is exactly a $\Pi^0_1$ subset of $\om$.
\item In ${\sf KV}$:
An open subobject of $\om^\om$ is exactly a $\Sigma^0_1$ subset of $\om^\om$, and a closed subobject of $\om^\om$ is exactly a $\Pi^0_1$ subset of $\om^\om$.
\item In ${\sf K}_2$:
An open subobject of $\om^\om$ is exactly an open subset of $\om^\om$, and a closed subobject of $\om^\om$ is exactly a closed subset of $\om^\om$.
\end{enumerate}
\end{example}

One can generalize these equivalences to (computable) Polish spaces and more, under appropriate (admissible) representations.


\subsection{Arithmetical hierarchy}

One of the main ingredients of the theory of hierarchies is the extraction of subsets using formulas.
This idea is common to both arithmetic and Borel hierarchies.

Let $\Gamma$ be a class of subobjects.
A sequence $(A_i)_{i\in I}$ of subobjects of $X$ is {\em uniformly $\Gamma$} if there exists a $\Gamma$-subobject $A\mono I\times X$ such that $A_i$ is the $i$th projection of $A$ for each $i\in I$.

\begin{definition}
Let $A$ be a subobject of a represented space $X$.
\begin{enumerate}
\item A subobject $A$ is ${\Sigma}^0_1$ if it is an open subobject of $X$.
\item A subobject $A$ is ${\Pi}^0_1$ if it is a closed subobject of $X$.
\item A subobject $A$ is ${\Sigma}^0_{n+1}$ if there exists a uniformly $\Pi^0_n$ sequence $(B_n)_{n\in\om}$ of subobjects of $X$ such that $A\equiv\biguplus_{n\in\om}B_n$.
\item A subobject $A$ is ${\Pi}^0_{n+1}$ if there exists a uniformly $\Sigma^0_n$ sequence $(B_n)_{n\in\om}$ of subobjects of $X$ such that $A\equiv\bignplus_{n\in\om}B_n$.
\end{enumerate}
\end{definition}

In ${\sf K}_2$, this definition coincides with the standard definition of the Borel hierarchy.
In ${\sf K}_1$, this is the arithmetical hierarchy with witnesses.

\begin{example}
Recall from Example \ref{exa:subobject-fin-def} that ${\sf Fin}$ is the subobject of $\om^\om$ consisting of sequences which are eventually zero.
Then ${\sf Fin}$ is a ${\Sigma}^0_2$ subobject of $\om^\om$.

This is because $F=\{(n,x)\in\om\times\om^\om:(\forall m\geq n)\;x(m)=0\}$ is a closed subobject of $\om\times\om^\om$; hence ${\sf Fin}=\exists^\om F=\biguplus_{n\in\om}F_n$ is $\Sigma^0_2$.
\end{example}

\begin{obs}\label{obs:basic-positive-Pi-0-2-regular}
Every ${\Pi}^0_2$ subobject of a represented space $X$ is regular.
\end{obs}

\begin{proof}
Every $\Pi^0_2$ subobject $A\mono X$ is of the form $\bignplus_{n\in\om}B_n$ for some open subobject $B\mono\om\times X$.
By Observation \ref{obs:open-closed-is-regular}, $B$ is regular, so a name of $(n,x)\in B$ is just the pair of $n\in\om$ and an $X$-name of $x$.
By definition, a name of $x\in A$ is a name of a function that, given $n\in\om$, returns an name of $(n,x)\in B$.
In particular, if $p$ is a name of $x\in X$ then the constant function $n\mapsto p$ gives a name of $x\in A$.
This means that $A\mono X$ is regular.
\end{proof}

\subsection{Internal logic}

It is also useful to introduce a method of defining a subobject of a represented space using a first order formula.
It is often easier to read the meaning of a construction of a subobject by using a formula than by combining set operations.

\begin{definition}
Given a sequence $(X_i)_{i\leq n}$ of represented spaces, let $x_i$ be a variable symbol of type $X_i$.
Put a subobject $R_i\mono X_i$ as a relation symbol of type $X_i$ in our language.
Then inductively define a subobject $\eval{\bar{x}:\varphi(\bar{x})}$ of $X_1\times\dots\times X_n$ as follows:
\begin{enumerate}
\item $\eval{(x_1,\dots,x_n):R_i(x_i)}=X_1\times\dots\times X_{i-1}\times R_i\times X_{i+1}\times\dots\times X_n$.
\item $\eval{\bar{x}:\varphi(\bar{x})\lor\psi(\bar{x})}=\eval{\bar{x}:\varphi(\bar{x})}\uplus\eval{\bar{x}:\psi(\bar{x})}$.
\item $\eval{\bar{x}:\varphi(\bar{x})\land\psi(\bar{x})}=\eval{\bar{x}:\varphi(\bar{x})}\nplus\eval{\bar{x}:\psi(\bar{x})}$.
\item $\eval{\bar{x}:\varphi(\bar{x})\to\psi(\bar{x})}=\eval{\bar{x}:\varphi(\bar{x})}\arr\eval{\bar{x}:\psi(\bar{x})}$.
\end{enumerate}

Moreover, for a represented space $I$,
\begin{enumerate}
\setcounter{enumi}{4}
\item $\eval{\bar{x}:\exists i\in I.\varphi(i,\bar{x})}=\biguplus_{i\in I}\eval{\bar{x}:\varphi(i,\bar{x})}$.
\item $\eval{\bar{x}:\forall i\in I.\varphi(i,\bar{x})}=\bignplus_{i\in I}\eval{\bar{x}:\varphi(i,\bar{x})}$.
\end{enumerate}
\end{definition}

Thereafter, a relation symbol $R(x)$ in a formula is sometimes abbreviated to $x\in R$.
Also, we sometimes use a notation such as $s(x)=t(y)$ in a formula, which determines the subspace whose underlying set is $\{(x,y):s(x)=t(y)\}$.

Hereafter, we mainly focus on the subobjects of $\om^\om$.
As already noted, in ${\sf K}_2$ and ${\sf KV}$, the object $\om^\om$ literally means the set of all infinite sequences of natural numbers, while in ${\sf K}_1$, the exponential object $\om^\om$ is the space of all total computable functions.
Thus, our theory on many-one reducibility for subobjects of $\om^\om$ in ${\sf K}_1$ correspond to the witnessed version of {\em many-one reducibility for index sets within ${\rm Tot}$} in classical terms, where ${\rm Tot}$ is the set of all indices of total computable functions on $\om$.

\section{The structure of $\Sigma^0_2$ sets}

\subsection{Union of closed sets}

According to Veldman \cite{VeldmanPhD}, in certain intuitionistic systems, interestingly, {\em the union of two $\Pi^0_1$ sets is not necessarily $\Pi^0_1$}.
We first see that this strange phenomenon can be given a clear interpretation even in classical logic by using reducibility for non-regular subobjects.

\begin{example}
Let $I$ be a represented space, and $X,Y\subseteq I$ be its subspaces.
Then the {\em tartan} $\uplus_I(X,Y)$ is defined as the witnessed union $(X\times I)\wunion(I\times Y)$.
In other words, $\uplus_{I}(X,Y)=\eval{(x,y):X(x)\lor Y(y)}$.
\end{example}

\begin{example}\label{exa:tartan-difference-hierarchy}
For the closed subspace $\{0^\infty\}\subseteq\om^\om$, consider the tartan ${\sf L}=\uplus_{\om^\om}(\{0^\infty\},\{0^\infty\})$.
In other words, ${\sf L}=\eval{(x,y):x=0^\infty\lor y=0^\infty}$.
\end{example}

Let $\Sigma_{\pi\cup\pi}$ be the class of subobjects which can be written as a witnessed union of two $\Pi^0_1$ sets.

\begin{obs}
${\sf L}\in\Sigma_{\pi\cup\pi}$.
\end{obs}

A typical way to show the non-regularity of a subobject in ${\sf K}_2$ or ${\sf KV}$ can be to use the following notion.

\begin{definition}\label{def:non-regular-point}
For a subobject $A\mono X$, we say that $x\in A$ is {\em non-regular} if, for some $\ell\in\om$, for any $A$-name $p$ of $x$ and $X$-name $q$ of $x$, there exists a sequence $(y_n)_{n\in\om}$ of points in $A$ such that
\begin{enumerate}
\item $p\upto\ell$ cannot be extended to an $A$-name of $y_n$ for any $n\in\om$.
\item for any $m$, $q\upto m$ can be extended to an $X$-name of $y_n$ for some $n\in\om$.
\end{enumerate}
\end{definition}

\begin{example}\label{exa:LLPO-has-non-regular-point}
${\sf L}$ has a non-regular point.
We claim that $z=(0^\infty,0^\infty)$ is a non-regular point of ${\sf L}$.
To see this, note that a name of $z\in{\sf L}$ is of the form $q_i=(i,0^\infty,0^\infty)$ for some $i<2$.
Then, put $\ell=1$, $y_m^0=(0^{m}1^\infty,0^\infty)$ and $y_m^1=(0^\infty,0^m1^\infty)$.
Then $q_i\upto\ell$ cannot be extended to an ${\sf L}$-name of $y_m^i$ for each $i<2$, while $z\upto m$ can be extended to $X$-names of both $y_m^0$ and $y_m^1$.
\end{example}

\begin{lemma}\label{lem:tartan-non-regular-general}
If a subobject $A\mono X$ has a non-regular point, then $A$ is not regular.
\end{lemma}

\begin{proof}
Assume that $A\mono X$ is regular.
Then the inclusion map $A\mono X$ has a partial left-inverse morphism, so let $r$ be its tracker.
That is, if $q$ is an $X$-name of $x\in|A|$, then $r(q)$ is an $A$-name of $x$.
Now, suppose that $A$ has a non-regular point $x\in A$, and let $\ell$ be a length for its non-regularity.
For an $X$-name $q$ of $x$ and an $A$-name $r(q)$ of $x$, we get a witness $(y_t)_{t\in\om}$ for non-regularity.
By continuity, after reading some $m$ bits of $q$, the first $\ell$ bits of $r(q)$ are determined.
Then $q\upto m$ can be extended to an $X$-name $q'$ of $y_t\in |A|$, so $r(q')$ must be an $A$-name of $y_t$.
However, $r(q')$ extends $r(q)\upto\ell$, which cannot be an $A$-name of $y_t$.
\end{proof}

\begin{cor}
A union of two $\Pi^0_1$ subobjects of $\om^\om$ is not necessarily $\Pi^0_1$.
Indeed, ${\sf L}$ is not ${\Pi}^0_2$, and thus $\Sigma_{\pi\cup\pi}\not\subseteq{\Pi}^0_2$.
\end{cor}

\begin{proof}
By Lemma \ref{lem:tartan-non-regular-general} and Example \ref{exa:LLPO-has-non-regular-point}, ${\sf L}$ is not regular.
However, by Observation \ref{obs:basic-positive-Pi-0-2-regular}, every ${\Pi}^0_2$ subobject is regular.
Hence, ${\sf L}$ is not ${\Pi}^0_2$.
\end{proof}

\subsection{Non-complete $\Sigma^0_2$ sets}

Classically, it is well-known that ${\rm Fin}=\{x\in\om^\om:(\exists n)(\forall m\geq n)\;x(m)=0\}$ is $\Sigma^0_2$-complete.
Surprisingly, Veldman \cite{Vel09} showed that this does not hold in certain intuitionistic systems.
His insightful work suggests that its witnessed counterpart is not $\Sigma^0_2$-complete in our setting.
\[{\sf Fin}:=\eval{x\in\om^\om:(\exists n)(\forall m\geq n)\;x(m)=0}=\biguplus_{n\in\om}\bignplus_{m\geq n}\{x\in\om^\om:x(m)=0\}\]

For the explicit description, recall Example \ref{exa:subobject-fin-def}.
Our first goal is to show the following:

\begin{theorem}\label{thm:fin-plus-incomplete}
${\sf Fin}$ is not $\Sigma^0_2$-complete; indeed, ${\sf L}\not\leq_{\sf m}{\sf Fin}$.
\end{theorem}

To prove this, let us introduce a few notions.
We say that a subobject $A\mono Z$ is {\em almost $\Pi^0_1$} if it is a finite witnessed union of $\Pi^0_1$ subobjects of $Z$.
%
We say that a subobject $A\mono Z$ is {\em amalgamable} if there exists $A'\equiv A$ and there exists a partial realizable function $F$ that, given a $Z$-name of $x\in|A'|$ and a finite sequence $(p_1,\dots,p_n)$, returns a name of $x\in A'$, whenever there exists $i\leq n$ such that $p_i$ is a name of $x\in A'$.

\begin{prop}\label{prop:fin-plus-amalgamable}
${\sf Fin}$ is an amalgamable subobject of $\om^\om$.
\end{prop}

\begin{proof}
Let $x\in|{\sf Fin}|$ and $(p_1,\dots,p_n)$ be given.
Assume that $p_i$ is a name of $x\in {\sf Fin}$ for some $i\leq n$.
As a name of an element of ${\sf Fin}$ is of the form $\langle s,z\rangle$, it does not lose generality to assume that each $p_i$ is of the form $\langle s_i,z_i\rangle$.
By our assumption, at least one of these is a correct name of $x\in {\sf Fin}$, so it is of the form $p_j=\langle s_j,x\rangle$ for some $j$.
Here, recall that $\langle s_j,x\rangle$ is a name of $x\in {\sf Fin}$ if and only if $x(t)=0$ for any $t\geq s_j$.
Note that, if we put $m=\max_{i\leq n}s_i$, then $\langle m,x\rangle$ is also a name of $x\in {\sf Fin}$.
This is because we have $s_j\leq m$, which implies $x(t)=0$ for any $t\geq m$.
In summary, $F(x,p_1,\dots,p_n)=\langle \max_{i\leq n}\pi_0(p_i),x\rangle$ gives a name of $x\in {\sf Fin}$.
Hence, ${\sf Fin}$ is amalgamable.
\end{proof}

\begin{lemma}\label{lem:almost-regular-amalgamable-regular}
If an almost $\Pi^0_1$ subobject $A\mono X$ is many-one reducible to an amalgamable subobject $B\mono Y$, then $A$ is regular.
\end{lemma}

\begin{proof}
Assume that $A\leq_{\sf m}B$ is witnessed by $\varphi\colon X\to Y$ and $r_-,r_+$.
To show that $A\mono Z$ is regular, suppose that an $X$-name $p$ of $x\in|A|$ is given.
Since $A$ is almost $\Pi^0_1$ and any $\Pi^0_1$ subobject is regular, we may assume that a name of $x\in A$ is of the form $(k,p)$ for some $k<n$.
Then $(r_-(k,p))_{k<n}$ is a sequence of candidates of names of $\varphi(x)\in B$.
That is, since $(k,p)$ is a correct name of $x\in A$ for some $k<n$, $r_-(k,p)$ must be a correct name of $\varphi(x)\in B$.
Here, note that $r_-(k,p)$ may be undefined, but the condition $x\in |A_k|$ is $\Pi^0_1(p)$, so once we see that this is refuted, we modify it to output some value.
In this way, we can assume that $r_-(k,p)$ is defined.

A tracker of $\varphi$ transforms an $X$-name of $x$ into an $Y$-name of $\varphi(x)$, which, together with the above candidates, can be used to obtain a correct $B$-name $q$ of $\varphi(x)$ by the assumption that $B$ is amalgamable.
Then, $r_+(p,q)$ is an $A$-name of $x\in A$.
This concludes that $A$ is regular.
\end{proof}

\begin{proof}[Proof of Theorem \ref{thm:fin-plus-incomplete}]
By Proposition \ref{prop:fin-plus-amalgamable}, ${\sf Fin}$ is amalgamable.
Obviously, ${\sf L}$ is almost $\Pi^0_1$.
Hence, by Lemma \ref{lem:almost-regular-amalgamable-regular}, if ${\sf L}\leq_{\sf m}{\sf Fin}$ then ${\sf L}$ is regular, which is impossible by Lemma \ref{lem:tartan-non-regular-general} and Example \ref{exa:LLPO-has-non-regular-point}.
\end{proof}

As in classical reducibility, one can see that ${\sf Fin}$ is $\Sigma^0_2$-complete w.r.t.~demi-many-one-reducibility.

\begin{obs}
For any $\Sigma^0_2$ subobject $A\mono\om^\om$, we have $A\leq_{\sf m}'{\sf Fin}$.
\end{obs}

\begin{proof}
The argument is similar to the classical $\Sigma^0_2$-completeness proof of ${\sf Fin}$.
Let $A=\eval{x:\exists n\forall m\theta(n,m,x)}$ be given, where $\theta$ is decidable.
Given $x$, we construct $\varphi(x)$.
In order to determine the value of $\varphi(x)(s)$, we first calculate the largest $n_s\leq s$ fulfilling the following condition:
For any $k\leq n_s$ there exists $m\leq s-n_s$ such that $\neg\theta(k,m,x)$.
If $n_s>n_{s-1}$ then put $\varphi(x)(s)=1$; otherwise $\varphi(x)(s)=0$.
Note that if $x\in A$ and if $n$ is the least witness for $x\in A$, then we have $n_s\leq n$ for any $s$.
Hence, for the least $s$ such that $n_s=n$, the value $s+1$ must be a witness for $\varphi(x)\in{\sf Fin}$.
Conversely, if $s$ is a witness for $\varphi(x)\in{\sf Fin}$, then $n_t=n_s$ for any $t\geq s$.
Then it is easy to see that $n_s+1$ is a witness for $x\in A$.
\end{proof}

The above proof shows that given the least witness for $x\in A$, one can compute a witness for $\varphi(x)\in{\sf Fin}$.
Of course, there may also be a non-least witness for $x\in A$, but if it is always possible to obtain the least witness for any $x\in A$, this suggests that $A$ be many-one reducible to ${\sf Fin}$.
Let us explore this idea further.

\subsection{Classification of $\Sigma^0_2$-complete sets}
A closer look at classical $\Sigma^0_2$-complete sets reveals that there are in fact various qualitative differences among them.
A $\Sigma^0_2$ set is a countable union of $\Pi^0_1$ sets, but there are various types of ``countable union,'' such as separated union, disjoint union, increasing union, and ordinary union.
In classical theory, separated union (coproduct) and ordinary union are distinguished, but others are not.

This situation can also be understood as a classification of $\Sigma^0_2$ formulas.
First note that ${\sf Fin}$ is defined by a $\Sigma^0_2$ formula of the form $\exists n\forall m\geq n.\varphi(m,x)$, where $\varphi$ is decidable.
For such a formula, if a witness $n$ is given, one can effectively find the smallest witness $n_0$ by checking $\varphi(m,x)$ holds for any $n_0\leq m\leq n$.
This means that such a formula can be replaced with a $\Sigma^0_2$ formula $\exists n\psi(n,x)$ having the following ``unique witness property'':
\[\exists n\psi(n,x)\iff\exists!n\psi(n,x).\]


Also, many $\Sigma^0_2$ sets can be written as an increasing union of $\Pi^0_1$ sets.
Expressed as a formula $\exists n\psi(n,x)$, this means that they have the following ``increasing witness property'':
\[m\leq n\;\&\;\psi(m,x)\implies\psi(n,x).\]

Of course, having the increasing witness property is precisely being defined by a $\Sigma^0_2$ formula of the form $\exists n\forall m\geq n.\varphi(m,x)$, where $\varphi$ is a $\Pi^0_1$ formula.

The notion of classical $\Sigma^0_2$-completeness makes no distinction between these special and general $\Sigma^0_2$-formulas.
The notion of many-one reducibility for nonregular subobjects helps to clarify these qualitative differences.

\begin{definition}
Let $A\mono X$ be a ${\Sigma}^0_2$ subobject.
\begin{enumerate}
\item $A$ has the {\em unique witness property} iff there exists a uniformly ${\Pi}^0_1$ sequence $(A_n)_{n\in\om}$ of subobjects of $X$ such that $A\equiv\biguplus_{n\in\om}A_n$ and $|A_n|\cap|A_m|=\emptyset$ whenever $n\not=m$.
\item $A$ has the {\em increasing witness property} iff there exists a uniformly ${\Pi}^0_1$ sequence $(A_n)_{n\in\om}$ of subobjects of $X$ such that $A\equiv\biguplus_{n\in\om}A_n$ and $|A_n|\subseteq|A_m|$ whenever $n\leq m$.
\end{enumerate}
\end{definition}


\begin{obs}\label{obs:unique-witness-formula-chara}
Let $A\mono X$ be a $\Sigma^0_2$ subobject such that $A=\eval{x\in X:\exists n\forall m\geq n.\;f(m,x)=1}$ for some morphism $f\colon \om\times X\to 2$.
Then $A$ has the unique witness property.
\end{obs}

\begin{proof}
Define $A_n=\eval{x\in X:\forall m\geq n.f(m,x)=1\land f(n-1,x)=0}$.
Then, $(A_n)_{n\in\om}$ is a uniform $\Pi^0_1$ sequence, and $|A_n|\cap |A_m|=\emptyset$ for $n\not=m$.
Clearly, we have $\biguplus_nA_n\subseteq A$.
To see $A\subseteq\biguplus_nA_n$, if $n$ is a witness for $x\in A$, then search for the least $n_0\leq n$ such that $f(m,x)=1$ for any $n_0\leq m\leq n$.
By our assumption on $n$, one can see that $n_0$ is the least witness for $x\in A$.
Therefore, $x\in A_{n_0}$.
As $n\mapsto n_0$ is computable, we have $A\subseteq\biguplus_nA_n$.
\end{proof}

\begin{obs}
A $\Sigma^0_2$ subobject $A\mono X$ has the increasing witness property iff there exists a $\Pi^0_1$ subobject $P\mono\om\times X$ such that $A=\eval{x\in X:\exists n\forall m\geq n.\;P(m,x)}$.
\end{obs}

\begin{proof}
For the forward direction, $x\in A$ iff $x\in A_n$ for some $n$, and for such $n$, $x\in A_m$ for any $m\geq n$.
For the backward direction, consider $A_n=\bignplus_{m\geq n}P_m$, where $P_m\mono X$ is the $m$th projection of $P$.
Then $(A_n)_{n\in\om}$ is a uniform $\Pi^0_1$ sequence, and $|A_n|\subseteq|A_m|$ whenever $n\leq m$.
Note that $n$ is a witness for $x\in A$ iff $x\in A_n$, so it is also a witness for $x\in\biguplus_{n\in\om}A_n$.
Thus, $A\equiv\biguplus_{n\in\om}A_n$.
\end{proof}

\begin{obs}\label{obs:sigma02-witnessproperty-downclosed}
Let $A\mono X$ and $B\mono Y$ be ${\Sigma}^0_2$ subobjects such that $A\leq_{\sf m}B$.
\begin{enumerate}
\item If $B$ has the unique witness property, so does $A$.
\item If $B$ has the increasing witness property, so does $A$.
\end{enumerate}
\end{obs}

\begin{proof}
(1)
Assume $B\equiv\biguplus_{n\in\om}B_n$.
If $A\leq_{\sf m}B$ via $\varphi$ then we have $A\equiv\varphi^\ast B$.
Note that $\varphi^\ast(\biguplus_nB_n)\equiv\biguplus_n\varphi^\ast B_n$.
Hence, $|B_n|\cap|B_m|=\emptyset$ implies $|\varphi^\ast B_n|\cap|\varphi^\ast B_m|=\varphi^{-1}[|B_n|]\cap\varphi^{-1}[|B_n|]=\varphi^{-1}[|B_n|\cap |B_m|]=\emptyset$.
Thus, $\varphi^\ast(\biguplus_nB_n)$ has the unique witness property.
Now, $B\equiv\biguplus_nB_n$ implies $A\equiv\varphi^\ast B\equiv\varphi^\ast(\biguplus_nB_n)$; hence, $A$ also has the unique witness property.

(2)
Note that $|B_n|\subseteq|B_m|$ implies $|\varphi^\ast B_n|=\varphi^{-1}[|B_n|]\subseteq\varphi^{-1}[|B_m|]=|\varphi^\ast B_m|$.
Thus, by the same argument as above, one can see that $A$ has the increasing witness property.
\end{proof}

\begin{prop}\label{prop:simga02-uniquewitness-to-increasingwitness}
If a ${\Sigma}^0_2$ subobject $A\mono X$ has the unique witness property then it has the increasing witness property.
\end{prop}

\begin{proof}
Assume $A\equiv\biguplus_{n\in\om}A_n$, and consider $A_m'=\biguplus_{k\leq m}A_k$.
Obviously, $m\leq n$ implies $|A_m'|\subseteq |A_n'|$.
One can see $A\equiv \biguplus_{m\in\om}A_m'$.
This is because, if $(n,p)$ is a name of $x\in \biguplus_{n\in\om}A_n$ then $(n,n,p)$ is a name of $x\in \biguplus_{m}\biguplus_{k\leq m}A_k$, and if $(m,k,p)$ is a name of $x\in \biguplus_{m}\biguplus_{k\leq m}A_k$ then $(k,p)$ is a name of $x\in \biguplus_{n\in\om}A_n$.

Now note that, in general, even if $A_n$ is ${\Pi}^0_1$ for each $n$, $A_m'$ is not necessarily ${\Pi}^0_1$.
However, the unique witness property solves this problem.
In this case, we claim that $A_m'$ is equivalent to a subspace $B_m$ of $X$ whose underlying set is $|A_m'|$.
Clearly, $A_m'\subseteq B_m$.
Given a name $p$ of $x\in B_m$, by the unique witness property, there exists a unique $k\leq m$ such that $x\in A_k$.
Now wait for $x\not\in A_n$ to be recognized for all $n\leq m$ except one $k$.
By continuity of trackers of co-characteristic functions of $A_n$'s, one can recognize this after reading a finite initial segment on $p$.
Then we must have $x\in A_k$.
Now, as $p$ is an $X$-name of $x$, by regularity, one can recover its $A_k$-name $p_k$.
Then $(k,p_k)$ is a name of $x\in A_m'$.

It remains to check that $B_m$ is a ${\Pi}^0_1$ subobject of $X$.
This is because $p$ is not a name of an element of $B_m$ iff $p_k$ is not a name of an element of $A_k$ for any $k\leq m$, which is recognizable.
\end{proof}

As an argument for modifying a given formula, let us introduce Veldman's notion of ``perhaps'' \cite{Vel05,Vel09}.
Given a $\Sigma^0_2$ formula $\psi\equiv\exists a\forall b\theta(a,b,x)$, consider the following:
\[{\sf Perhaps}(\psi):=\eval{x\in\om^\om:\exists a[\forall b\;(\neg\theta(a,b,x)\to\exists c\forall d\theta(c,d,x)]}\]

Note that the above has the same many-one degree as $\eval{x:\exists a[\eta(a,x)\to\exists b\neg\eta(b,x)]}$, where $\eta(a,x)\equiv\exists z\neg\theta(a,z,x)$.
Veldman \cite[Theorem 3.8]{Vel05} proved that ${\sf Perhaps}({\sf Fin})$ is $\Sigma^0_4$ but not $\Pi^0_4$, and in particular, it jumps out of $\Sigma^0_2$.
Here, we consider its totalized version that falls into the framework of $\Sigma^0_2$.
\[{\sf Half}(\psi):=\eval{x\in\om^\om:\exists a,b[\eta(a,x)\to\neg\eta(b,x)]}\]

To be more explicit, for a $\Sigma^0_2$ formula $\psi\equiv\exists a\forall z\theta(a,z,x)$, the subobject ${\sf Half}(\psi)\mono\om^\om$ is given as follows:
\begin{itemize}
\item The underlying set is $\{x\in\om^\om:\psi(x)\}$.
\item $(a,b,x)$ is a name of $x\in{\sf Half}(\psi)$ iff either $\forall z\theta(a,z,x)$ or $\forall z\theta(b,z,x)$ holds.
\end{itemize}

\begin{definition}
A subobject $A\mono\om^\om$ is {\em half $\Sigma^0_2$-hard} if ${\sf Half}(\psi)\leq_{\sf m}A$ for any $\Sigma^0_2$ formula $\psi$.
\end{definition}

Later we will show that no half $\Sigma^0_2$-hard subobject has the increasing witness property.

\subsection{Examples}
We give some natural examples of witnessed sets whose underlying sets are classically $\Sigma^0_2$-complete.
The examples listed below seem to fall into four groups.

\subsubsection{Unique witness property}
The first group consists of examples of the ``unique witness'' type.


\begin{example}
A binary relation $R\subseteq\om\times\om$ is coded by its characteristic function $\chi_R\in 2^{\om\times\om}$.
Let $(P_x,\leq_x)\in{\sf PO}$ denote the partial order coded by $x\in\om^\om$, where $P_x=\{a\in\om:a\leq_xa\}$.
In this way, the space ${\sf PO}$ of partial orders on a subset of $\om$ can be introduced as a $\Pi^0_1$ subspace of $2^{\om\times\om}$.
Consider the following subobject ${\sf PO}_{\sf top}$ of $\om^\om$:
\begin{itemize}
\item The underlying set is the set of all partial orders having greatest elements; that is, $|{\sf PO}_{\sf top}|=\{x\in\om^\om:(\exists a\in P_x)(\forall b\in P_x)\;b\leq_xa\}$.
\item $(a,p)$ is a name of $x\in{\sf PO}_{\sf top}$ iff $p=x$ and $a$ is the greatest element in $(P_x,\leq_x)$.
\end{itemize}
\end{example}

Since the greatest element is unique if it exists, it is obvious that it has the unique witness property.

\begin{prop}
${\sf Fin}\equiv_{\sf m}{\sf PO}_{\sf top}$.
\end{prop}

\begin{proof}
${\sf PO}_{\sf top}\leq_{\sf m}{\sf Fin}$:
Given $P=P_x$, consider $P[s]=\{p\leq s:p\leq_Pp\}$.
If $\max P[s]=\max P[s+1]$ then put $\varphi(x)(s)=0$; otherwise $\varphi(x)(s)=1$.
If $p$ is the $\leq_P$-greatest element, then we have $p=\max P[p]=\max P$, so $\varphi(x)(s)=0$ for any $s\geq p$.
Hence, $p$ is a witness for $\varphi(x)\in{\sf Fin}$.
Conversely, if $s$ is a witness for $\varphi(x)\in{\sf Fin}$, then $\varphi(x)(t)=0$ for any $t\geq s$.
This means $\max P[s]=\max P$, so compute the $\leq_P$-greatest element $p$ among the finite set $P[s]$.
Then $p$ is a witness for $x\in{\sf PO}_{\sf top}$.
%

${\sf Fin}\leq_{\sf m}{\sf PO}_{\sf top}$:
Given $x\in\om^\om$, we construct a poset $P=P_{\varphi(x)}$.
Whenever $x(s)\not=0$ happens, we add a new top element to $P$; that is, if $x(s)\not=0$, put $t<_Ps\leq_Ps$ for any $t<s$; otherwise we do not add anything to $P$.
If $s$ is a witness for $x\in{\sf Fin}$ then search for the $\leq_P$-greatest element $p$ among the finite set $\{t\in P:t\leq s\}$.
Since nothing is added to $P$ after stage $s$, this $p$ remains the $\leq_P$-greatest element in $P$, so $p$ is a witness for $\varphi(x)\in {\sf PO}_{\sf top}$.
Conversely, if $p$ is the $\leq_P$-greatest element in $P$ then there is no $s>p$ such that $x(s)\not=0$.
Otherwise, we add a new top element $s>_Pp$, which is impossible.
Thus, $p$ is a witness for $x\in{\sf Fin}$.
\end{proof}

Of course, partial orders may be changed to linear orders, and top elements may be changed to bottom elements.
This means that the decision of boundedness ${\sf BddPO}$ of posets is also ${\sf m}$-equivalent to ${\sf Fin}$.
Here are some other examples.

\begin{example}\label{exa:sigma-02-complete1}
Consider the following subobject ${\sf Conv}$ of $\om^\om$:
\begin{itemize}
\item The underlying set is the set of all convergent sequences on $\om$; that is, $|{\sf Conv}|=\{x\in\om^\om:\lim_{n\to\infty}x(n)\mbox{ exists}\}$.
\item $(s,p)$ is a name of $x\in{\sf Conv}$ iff $p=x$ and $x(t)=x(s)$ for all $t\geq s$.
\end{itemize}
\end{example}

\begin{example}\label{exa:sigma-02-complete2}
We consider a decision problem for real numbers, where a real number is presented by accuracy-guaranteed rational approximations, so we deal with a decision problem for such rational approximations.
A pre-real is a rational sequence $(q_n)_{n\in\om}$ such that $|q_n-q_m|\leq 2^{-n}$ for any $m\geq n$.
Here, rational numbers are coded by natural numbers in an obvious manner, so we can consider $\mathbb{R}_{\sf pre}$ to be a $\Pi^0_1$ subobject of $\om^\om$.
In the following, we often identify a pre-real $(q_n)_{n\in\om}$ with the real number $\lim_{n\to\infty}q_n$.
If the limit is rational, we say that it is a pre-rational.
Consider the following subobject $\mathbb{Q}_{\sf pre}$ of $\mathbb{R}_{\sf pre}$:
\begin{itemize}
\item The underlying set is the set of all pre-rationals.
\item $(m,n,p)$ is a name of $x\in\mathbb{Q}_{\sf pre}$ iff $p$ is an $\mathbb{R}_{\sf pre}$-name of $x$, and $x=\frac{k}{m}$ where $m\in\om\setminus\{0\}$, and $n$ is a $\mathbb{Z}$-name of $k$.
\end{itemize}
\end{example}

Here, for example, although there can be more than one witness $\frac{k}{m}$ for $x\in\mathbb{Q}_{\sf pre}$, if restricted to only irreducible fractions, the witness for $x\in\mathbb{Q}_{\sf pre}$ is unique.
Thus, $\mathbb{Q}_{\sf pre}\mono\mathbb{R}_{\sf pre}$ can be described by a formula with the unique witness property without changing the intrinsic complexity.
In this sense, $\mathbb{Q}_{\sf pre}\mono\mathbb{R}_{\sf pre}$ is an object of the unique witness type.

The case of ${\sf Conv}$ is a little more difficult.
At first glance, ${\sf Conv}$ appears to be of the increasing type, but closer analysis reveals that it is in fact of the unique witness type.
We will discuss this in detail later, but first let us see the following:

\begin{obs}
${\sf Fin}\equiv_{\sf m}{\sf Conv}\equiv_{\sf m}\mathbb{Q}_{\sf pre}$.
\end{obs}

\begin{proof}
${\sf Conv}\leq_{\sf m}{\sf Fin}$:
Given $x\in\om^\om$, define $\varphi(x)\in\om^\om$ as follows:
If $x(s+1)\not=x(s)$, put $\varphi(x)(s)=1$; otherwise $\varphi(x)(s)=0$.
Note that $s$ is a witness for $x\in{\sf Conv}$ iff $x(t+1)=x(t)$ for any $t\geq s$ iff $\varphi(x)(t)=0$ for any $t\geq s$ iff $s$ is a witness for $\varphi(x)\in{\sf Fin}$.
Thus, using $r_-(s,x)=s$ and $r_+(s,x)=s$ work.

${\sf Fin}\leq_{\sf m}\mathbb{Q}_{\sf pre}$:
For $f(n)=n^2$, the sum $\sum_{n=0}^\infty 2^{-f(n)}$ is irrational.
Given $x\in\om^\om$, consider $\varphi(x):=\sum_{x(n)\not=0}2^{-f(n)}$.
Note that $x\in{\sf Fin}$ iff the binary expansion of $\varphi(x)$ is eventually periodic iff $\varphi(x)$ is rational.
Given a witness $s$ for $x\in{\sf Fin}$, the value $\varphi(x)$ can be written as $\sum_{n\leq s}2^{f(s)-f(n)}/2^{f(s)}$.
The denominator and numerator are both natural numbers, so the pair is a witness for $\varphi(x)\in\mathbb{Q}_{\sf pre}$.
Conversely, if $\varphi(x)$ is of the form $a/b$, the denominator $b$ must be a multiple of some $2^{f(s)}$.
If $x(t)\not=0$ for some $t>s$, then $2^{-f(t)}$ is added to $\varphi(x)$, which makes it impossible to express $\varphi(x)$ as a multiple of $2^{-f(s)}$.
Thus, $s+1$ is a witness for $x\in{\sf Fin}$.

$\mathbb{Q}_{\sf pre}\leq_{\sf m}{\sf Conv}$:
Given $x\in\mathbb{R}_{\sf pre}$, we first get its rational approximation with accuracy $1$, from which we can compute a positive integer $b$ that is an upper bound of $|x|$, so we get $y=x/b\in[-1,1]$.
We define $\varphi(x)(s)$ as the prediction of the denominator of $y$ at stage $s$; that is, the current prediction is $n$ iff the prediction that $y=k/n$ for some $-n\leq k\leq n$ is not refuted at that stage.
Here, the prediction that $y=k/n$ is refuted at stage $s$ means that $|q_s-k/n|>2^{-s}$ is confirmed by looking at the information on a rational approximation $q_s$ of $y$ with accuracy $2^{-s}$.
If the prediction $n$ is refuted at $s$, change the prediction at stage $s+1$ to $\varphi(x)(s+1)=n+1$.

If $x$ is a rational number of the form $a/b$, then rewrite this into irreducible fraction $k/n$.
One can see that the denominator $n$ is equal to $\lim_{s\to\infty}\varphi(x)(s)$.
We search for the first $s$ such that $\varphi(x)(s)=n$.
This $s$ is a witness for $\varphi(x)\in{\sf Conv}$ since $\varphi(x)$ is monotone.
Conversely, if $s$ is a witness for $\varphi(x)\in{\sf Conv}$, then compute $\varphi(x)(s)=n$.
By our construction, $x$ must be of the form $k/n$ for some $-n\leq k\leq n$.
Looking at an approximation of $x$ with accuracy $2^{-n}$, the equation $x=k/n$ is refuted except for one $k$, so the last remaining $k$ is the numerator of $x$.
\end{proof}

The reason why we introduced the notion of pre-real here is in the proof of $\mathbb{Q}_{\sf pre}\leq_{\sf m}{\sf Conv}$.
By our definition, a many-one reduction $\varphi$ must be well-defined on represented spaces; however, observe that there are few morphisms $\mathbb{R}\to\om^\om$, while there are many morphisms $\mathbb{R}_{\sf pre}\to\om^\om$.
In the context of Wadge reducibility (topological many-one reducibility; Definition \ref{def:Wadge-reducibility}), the difference between the structures of $\mathbb{R}$ and $\om^\om$ is examined in depth \cite{Sch18}.
If we change the definition to a form that allows a many-one reduction $\varphi$ on the name space ${\tt Code}$, as in Pequignot-style Wadge reducibility \cite{Peq15} or Weihrauch reducibility \cite{pauly-handbook}, there is no need to introduce the notion of pre-real.

%
%
%

\subsubsection{Increasing witness property}
Next, let us discuss examples that are of the ``increasing'' type.

\begin{example}\label{exa:sigma-02-complete4}
Consider the following subobject ${\sf BddSeq}_\om$ of $\om^\om$:
\begin{itemize}
\item The underlying set is the set of all bounded sequences on $\om$; that is, $|{\sf BddSeq}_\om|=\{x\in\om^\om:(\exists b)(\forall n)\;x(n)\leq b\}$.
\item $(b,p)$ is a name of $x\in{\sf BddSeq}_\om$ iff $p=x$ and $x(n)\leq b$ for all $n$.
\end{itemize}

${\sf BddSeq}_\mathbb{Q}\mono\mathbb{Q}^\om$ and ${\sf BddSeq}_\mathbb{R}\mono\mathbb{R}^\om$ can be defined in a similar manner.
\end{example}

It is not difficult to verify that these examples can be written as increasing sequences of $\Pi^0_1$ sets. 

\begin{obs}
${\sf BddSeq}_\om\equiv_{\sf m}{\sf BddSeq}_\mathbb{Q}\equiv_{\sf m}{\sf BddSeq}_\mathbb{R}$.
\end{obs}

\begin{proof}
A natural number is clearly a rational number, which is clearly a real number.
Given an upper bound $b$ of a sequence $(x_n)_{n\in\om}$ of real numbers, first extract a rational approximation $m/n$ of $b$ with precision $2^{-1}$.
Then the natural number $|m|+1$ is an upper bound of $(x_n)_{n\in\om}$.
\end{proof}

An example other than decision problems on sequences is the decision that a partial order has a finite height/width.
Here, the {\em height} ({\em width}, resp.)~of a poset $P$ is the supremum of the cardinality of chains (antichains, resp.)~in $P$.

If one attempts to describe finiteness of the size of something by a $\Sigma^0_2$-formula, it is natural to write it as the existence of a finite upper bound of the size.
Therefore, it is appropriate to consider the witness of this formula as the value of an upper bound of the size.

\begin{example}
Recall that $P_x$ is the poset coded by $x\in\om^\om$.
Consider the following subobjects ${\sf FinHeight}$ and ${\sf FinWidth}$ of $\om^\om$:
\begin{itemize}
\item $|{\sf FinHeight}|=\{x\in\om^\om:\mbox{the height of $P_x$ is finite}\}$.
\item $|{\sf FinWidth}|=\{x\in\om^\om:\mbox{the width of $P_x$ is finite}\}$.
\item $(b,p)$ is a name of $x\in{\sf FinHeight}$ iff $p=x$ and the height of $P_x$ is at most $b$.
\item $(b,p)$ is a name of $x\in{\sf FinWidth}$ iff $p=x$ and the width of $P_x$ is at most $b$.
\end{itemize}
\end{example}

\begin{obs}
${\sf BddSeq}_\om\equiv_{\sf m}{\sf FinHeight}\equiv_{\sf m}{\sf FinWidth}$.
\end{obs}

\begin{proof}
${\sf FinHeight},{\sf FinWidth}\leq_{\sf m}{\sf BddSeq}_\om$:
Given $P=P_x$, consider the cardinality $\varphi(x)(n)$ of a maximal chain or antichain in the finite set $P[n]=\{p\leq n:p\leq_Pp\}$.

${\sf BddSeq}_\om\leq_{\sf m}{\sf FinHeight}$:
Given $x\in\om^\om$, we construct a poset $P=P_{\varphi(x)}$.
Whenever a previously unseen large value appears in $x$, we add a new top element to $P$.
To be precise, let the underlying set of $P$ be $\{\pair{n,t}:x(t)\geq n\mbox{ and }\forall s<t.\;x(s)<n\}$.
Here, $n<m$ implies $\pair{n,t}<_P\pair{m,s}$ for any $s,t$.
Note that for each $n$, $P$ has at most one element of the form $\pair{n,t}$.

If $b$ is a witness for $x\in{\sf BddSeq}_\om$, then $x(t)\leq n$ for any $t$.
Thus, $\pair{m,s}\in P$ implies $m\leq n$.
Then $P$ can contain at most $n+1$ elements as noted above, so the height of $P$ is at most $n+1$.
Therefore, $n+1$ is a witness for $\varphi(x)\in{\sf FinHeight}$.
Conversely, if $n$ is a witness for $\varphi(x)\in{\sf FinHeight}$ then we claim that $n$ is a witness for $x\in{\sf BddSeq}_\om$.
Otherwise, there exists $t$ such that $x(t)>n$.
Hence, for any $m\geq n$, let $t(m)$ be the least number such that $x(t(m))\geq n$.
By definition, we have $\pair{m,t(m)}\in P$.
In particular, $(a^m_{t(m)})_{m\leq n}$ is a chain of length $n+1$; thus the height of $P$ is at least $n+1$.
This contradicts our assumption, so $n$ is a witness for $x\in{\sf BddSeq}_\om$.

${\sf BddSeq}_\om\leq_{\sf m}{\sf FinWidth}$:
Given $x\in\om^\om$, we construct a poset $P=P_{\varphi(x)}$.
The underlying set is the same as above.
Now we declare that all distinct elements in $P$ are incomparable.
The rest of the proof is exactly the same as above.
\end{proof}

\subsubsection{Half $\Sigma^0_2$-hard}
Let us discuss $\Sigma^0_2$ subobjects that are half $\Sigma^0_2$-hard.
Obviously, there exists a subobject of $\om^\om$ which is complete w.r.t.~the ones of the form ${\sf Half}(\psi)\mono\om^\om$ for some $\Sigma^0_2$ formula $\psi$.
To see this, just take a universal $\Sigma^0_2$ formula $\psi$.

\begin{example}
Consider the following subobject ${\sf HalfTruth}_{\Sigma^0_2}\mono\om^\om$:
\begin{itemize}
\item The underlying set is $\{(x_n)_{n\in\om}\in(\om^\om)^\om:(\exists n\in\om)\;x_n=0^\infty\}$.
\item $(n,m,p)$ is a name of $x=(x_n)_{n\in\om}\in{\sf DisConn}$ iff $p=x$ and either $x_n=0^\infty$ or $x_m=0^\infty$ holds.
\end{itemize}
\end{example}

\begin{obs}
${\sf HalfTruth}_{\Sigma^0_2}$ is half $\Sigma^0_2$-hard.
\end{obs}

However, this is just a meta-mathematical example of a half $\Sigma^0_2$-complete subobject.
A natural example of a half $\Sigma^0_2$-hard subobject may be presented from graph theory.

There are two ways to formalize the notion of a graph.
One is to introduce edges as pairs of vertices.
That is, a directed graph is a pair $G=(V,E)$ satisfying $E\subseteq V^2$.
We call this a {\em subset presentation} of a directed graph.

The other is to treat edges as data with specified source and target vertices.
In other words, a directed multigraph is a map $\langle d,c\rangle\colon E\to V^2$, where $d(e)$ denotes the source vertex (tail) of edge $e$ and $c(e)$ denotes the target vertex (head) of edge $e$.
This is also a very standard way of presenting a directed multigraph, which we call a {\em function presentation} of a graph.
We sometime use the symbol $u\edge{e}v$ to denote an edge $e$ with $d(e)=u$ and $c(e)=v$.

From here on, we only deal with undirected graphs.
In the case of a presentation of an undirected graph, we consider $E\subseteq[V]^2$ and $\gamma\colon E\to[V]^2$.
Here, $[V]^2=\{(u,v)\in V^2:u<v\}$ for $V\subseteq\om$, and each $(u,v)\in[V]^2$ is often written as $\{u,v\}$ or $\{v,u\}$; that is, we do not distinguish between $\{u,v\}$ and $\{v,u\}$ as usual.

Note that subset and function presentations (even restricted to undirected simple graphs) are not computability-theoretically equivalent in a certain sense, but they are equivalent as far as the following disconnectedness problem is concerned.

\begin{example}
A subset presentation of a graph $G=(V,E)$ with $V\subseteq\om$ and $E\subseteq[V]^2$ is coded as their characteristic functions $\chi_V\in 2^\om$ and $\chi_E\in 2^{(\om^2)}\simeq 2^\om$.
Let $G_x=(V_x,E_x)$ be the subset presentation of the graph coded by $x$.
Consider the following subobject ${\sf DisConn}$ of $\om^\om$:
\begin{itemize}
\item The underlying set is the set of all subset presentations of disconnected graphs; that is, $x\in|{\sf DisConn}|$ iff not every two vertices are connected by a finite path in $G_x$.
\item $(a,b,p)$ is a name of $x\in{\sf DisConn}$ iff $p=x$ and $a,b\in V_x$ are not connected by any finite path in $G_x$.
\end{itemize}
\end{example}

\begin{example}
A function presentation of a graph $\gamma\colon E\to [V]^2$ with $V\subseteq\om$ and $E\subseteq\om$ is coded as the triple $(\gamma,\chi_V,\chi_E)$.
Let $\gamma_x\colon E_x\to[V_x]^2$ be the function presentation of the graph coded by $x$.
Consider the following subobject ${\sf DisConn}_{\sf fun}$ of $\om^\om$:
\begin{itemize}
\item The underlying set is the set of all function presentations of disconnected graphs; that is, $x\in|{\sf DisConn}_{\sf fun}|$ iff not every two vertices are connected by a finite path in $\gamma_x$.
\item $(a,b,p)$ is a name of $x\in{\sf DisConn}_{\sf fun}$ iff $p=x$ and $a,b\in V_x$ are not connected by any finite path in $\gamma_x$.
\end{itemize}
\end{example}

These are clearly $\Sigma^0_2$ subobjects of $\om^\om$.
A function presentation of a graph may appear, for example, in the context of group actions.

\begin{example}
An action $\alpha\colon G\times S\to S$ of a countable group $G\subseteq\om$ on a set $S\subseteq\om$ is coded via their characteristic functions, where a code of $G\subseteq\om$ also involves the operation $\ast\in \om^{\om\times\om}$ and the inverse $\circ^{-1}\in\om^\om$.
Let $(G_x,S_x,\alpha_x)$ denote the countable group action coded by $x\in\om^\om$.
Consider the following subobject ${\sf Orbit}_{\geq 2}$ of $\om^\om$:
\begin{itemize}
\item The underlying set is the set of all countable group actions on subsets of $\om$; that is, $x\in|{\sf Orbit}_{\geq 2}|$ iff $\alpha_x\colon G_x\times S_x\to S_x$ has at least two orbits.
\item $(a,b,p)$ is a name of $x\in{\sf DisConn}_{\sf fun}$ iff $p=x$ and $a,b\in S_x$ belong to different orbits; that is, there is no $g\in G_x$ such that $\alpha_x(g,a)=b$.
\end{itemize}

When a group action $\alpha$ is given, $\alpha(g,a)$ is often abbreviated as $g\cdot a$.
\end{example}

\begin{prop}
${\sf DisConn}\equiv_{\sf m}{\sf DisConn}_{\sf fun}\equiv_{\sf m}{\sf Orbit}_{\geq 2}$.
\end{prop}

\begin{proof}
${\sf DisConn}\leq_{\sf m}{\sf DisConn}_{\sf fun}$:
It is obvious since any subset $E\subseteq[V]^2$ can be thought of as an inclusion map $E\embed[V]^2$.

${\sf DisConn}_{\sf fun}\leq_{\sf m}{\sf DisConn}$:
Let a function presentation of a graph $\gamma\colon E\to[V]^2$ is given.
For each $v\in V$, put $2v\in V'$.
If $\gamma(a)=\{u,v\}$ then put $\{2u,2\langle u,v,a\rangle+1\},\{2\langle u,v,a\rangle+1,2v\}\in E'$.
Note that the graph $\phi(\gamma)=(V',E')$ is the result of adding one vertex to the midpoint of each edge of the graph $(V,E)$.
For the forward reduction, clearly $\{u,v\}$ is disconnected in $(V,E)$ iff $\{2u,2v\}$ is disconnected in $(V',E')$.
For the backward reduction, assume that $\{2u,2\langle v,w,a\rangle+1\}$ is disconnected in $(V',E')$.
By definition, $2v$ is adjacent to $2\langle v,w,a\rangle+1$, so $\{2u,2v\}$ must also be disconnected; hence $\{u,v\}$ is disconnected in $(V,E)$.
Similarly, if $\{2\langle u,v,a\rangle+1,2\langle u',v',a'\rangle+1\}$ is disconnected in $(V',E')$, so is $(2u,2u')$; hence $(u,u')$ is disconnected in $(V,E)$.

${\sf Orbit}_{\geq 2}\leq_{\sf m}{\sf DisConn}_{\sf fun}$:
Given a group action $G\times S\to S$, consider the (directed) graph $\gamma\colon G\times S\to S^2$ defined by $\gamma(g,a)=(a,g\cdot a)$.
Intuitively, $\gamma$ consists of edges of the form $a\edge{g}g\cdot a$, but this must be distinguished from $b\edge{g}g\cdot b$ for $b\not=a$, so we add the information on tails to the names of these edges; that is, the name of the former is $(g,a)$ and the latter is $(g,b)$.
Now every $g\in G$ is invertible, so we have $\gamma(g^{-1},g\cdot a)=(g\cdot a,a)$.
As the information of $\circ^{-1}$ is contained in a code of $G$, given an edge $u\edge{e}v$ in $\gamma$, one can always effectively find an edge $v\edge{c}u$ in $\gamma$; that is, if $e=(g,a)$ then $c=(g^{-1},g\cdot a)$.
Hence $\gamma$ can be modified to an undirected graph $\gamma'\colon G\times S\to [S]^2$.
Now, it is easy to see that $a,b\in S$ belong to the same orbit iff $a$ and $b$ are connected by a finite path in $\gamma'$.

${\sf DisConn}_{\sf fun}\leq_{\sf m}{\sf Orbit}_{\geq 2}$:
Given a function presentation of a graph $\gamma\colon E\to[V]^2$, consider the free group $F_E$ over the set $E$.
Then define the $F_E$-action on $V$ as follows:
For $a\in E$, if $\gamma(a)=\{u,v\}$ then put $a\cdot u=v$ and $a\cdot v=u$.
For $w\in V\setminus\{u,v\}$, put $a\cdot w=w$.
Then $a$ and $b$ are connected by a finite path in $\gamma$ iff $a,b\in S$ belong to the same orbit.
\end{proof}

Interestingly, as we will see later, ${\sf DisConn}$ is half $\Sigma^0_2$-hard, but not $\Sigma^0_2$-complete.

\subsubsection{$\Sigma^0_2$-complete}
The last group consists of ``genuine'' $\Sigma^0_2$-complete sets.
Of course, a meta-mathematical example of a $\Sigma^0_2$-complete subobject of $\om^\om$ is $\eval{x\in\om^\om:\exists a\in\om\forall b\in\om\theta(a,b,x)}$, where $\exists a\forall b\theta(a,b,x)$ is a universal $\Sigma^0_2$ formula.
This is equivalent to the following:

\begin{example}
Consider the following subobject ${\sf Truth}_{\Sigma^0_2}\mono\om^\om$:
\begin{itemize}
\item The underlying set is $\{(x_n)_{n\in\om}\in(\om^\om)^\om:(\exists n\in\om)\;x_n=0^\infty\}$.
\item $(n,p)$ is a name of $x=(x_n)_{n\in\om}\in{\sf Truth}_{\Sigma^0_2}$ iff $p=x$ and $x_n=0^\infty$ holds.
\end{itemize}
\end{example}

\begin{obs}
${\sf Truth}_{\Sigma^0_2}$ is complete w.r.t.~$\Sigma^0_2$ subobjects of $\om^\om$.
\end{obs}

Next, we give some order-theoretic examples of $\Sigma^0_2$-complete subobjects.

\begin{example}
The space of linear orders ${\sf LO}$ on a subset of $\om$ is a $\Pi^0_1$ subspace of $2^{\om\times\om}$.
Let $(L_x,\leq_x)\in{\sf LO}$ denote the linear order coded by $x\in\om^\om$, where $L_x=\{a\in\om:a\leq_xa\}$.
Consider the following subobject ${\sf NonDense}$ of $\om^\om$:
\begin{itemize}
\item The underlying set is the set of all non-dense linear orders; that is, $|{\sf NonDense}|=\{x\in\om^\om:\neg(\forall a,b\in L_x)\;[a<_xb\to (\exists c\in L_x)\;a<_xc<_xa]\}$.
\item $(a,b,p)$ is a name of $x\in{\sf NonDense}$ iff $p=x$, $a<_xb$ and for any $c\in P_x$ either $c\leq_xa$ or $b\leq_xc$ holds.
\end{itemize}
\end{example}

\begin{example}
A bottomed poset is a tuple $(P,\leq_P,\bot_P)$, where $(P,\leq_P)$ is a poset, and $\bot_P$ is the least element in $P$.
An atom of a bottomed poset $P$ is an element which is minimal among non-bottom elements in $P$.
For $P\subseteq\om$, a bottomed poset $(P,\leq_P,\bot_P)$ is coded by the pair of the characteristic function of $\leq_P$ and the natural number $\bot_P\in P\subseteq\om$.
In this way, the space of bottomed posets ${\sf PO}_\bot$ on a subset of $\om$ can be introduced as a $\Pi^0_1$ subspace of $2^{\om\times\om}\times\om$.

Let $(P_x,\leq_x,\bot_x)\in{\sf PO}_\bot$ denote the bottomed poset coded by $x\in\om^\om$, where $P_x=\{a\in\om:a\leq_xa\}$.
Consider the following subobject ${\sf Pos}_{\sf atom}$ of $\om^\om$:
\begin{itemize}
\item The underlying set is the set of all partial orders having atoms; that is, $|{\sf PO}_{\sf atom}|=\{x\in\om^\om:(\exists a\in P_x)(\forall b\in P_x)\;b<_xa\to b=\bot_x\}$.
\item $(a,p)$ is a name of $x\in{\sf PO}_{\sf atom}$ iff $p=x$, for any $c\in P_x$ either $c\leq_xa$ or $b\leq_xc$ holds.
\end{itemize}
\end{example}

\begin{prop}
${\sf NonDense}$ is complete w.r.t.~$\Sigma^0_2$ subobjects of $\om^\om$.
\end{prop}

\begin{proof}
It suffices to show ${\sf Truth}_{\Sigma^0_2}\leq_{\sf m}{\sf NonDense}$.
Given $x=(x_n)_{n\in\om}$, we construct a linear order $L=L_{\varphi(x)}$.
First put $(n,0)<_L(n+1,0)$ for any $n\in\om$.
As long as $x_n=0^\infty$ is true, nothing is enumerated between $(n,0)$ and $(n+1,0)$.
If we see $x_n\not=0^\infty$ at stage $s$, we enumerate elements of the form $(n,t)$ for $t\geq s$ between $(n,0)$ and $(n+1,0)$ so that the $\leq_L$-interval $[(n,0),(n+1,0)]$ eventually becomes dense.

If $n$ is a witness for $x\in{\sf Truth}_{\Sigma^0_2}$, i.e., $x_n=0^\infty$, then there is no element between $(n,0)$ and $(n+1,0)$, so this pair is a witness for $\varphi(x)\in{\sf NonDense}$.
Conversely, let $\pair{(n,i),(m,j)}$ be a witness for $\varphi(x)\in{\sf NonDense}$.
We may assume $n\leq m$.
If $n+1<m$ then we have $(n,i)<(n+1,0)<(m,j)$, so $\pair{(n,i),(m,j)}$ cannot be a witness.
If $n=m$ and $i\not=j$, then either $i\not=0$ or $j\not=0$; that is, some $(n,k)$ for $k\not=0$ is enumerated into $L$.
By our construction, this means that the $\leq_L$-interval $[(n,0),(n+1)]$ becomes dense, and any element of the form $(n,k)$ is enumerated into this interval.
In particular, $(n,i)$ and $(n,j)$ are contained in the $\leq_L$-dense interval $[(n,0),(n+1)]$, so some $(n,k)$ is enumerated between $(n,0)$ and $(n+1,0)$; hence $\pair{(n,i),(m,j)}$ cannot be a witness.
Therefore, we have $m=n+1$, but there is no element between $(n,i)$ and $(n+1,j)$ then we must have $i=j=0$, which means $x_n=0^\infty$; that is, $n$ is a witness for $x\in{\sf Truth}_{\Sigma^0_2}$.
\end{proof}

\begin{prop}
${\sf PO}_{\sf atom}$ is complete w.r.t.~$\Sigma^0_2$ subobjects of $\om^\om$.
\end{prop}

\begin{proof}
It suffices to show ${\sf Truth}_{\Sigma^0_2}\leq_{\sf m}{\sf PO}_{\sf atom}$.
Given $x=(x_n)_{n\in\om}$, we construct a partially ordered set $P=P_\varphi(x)$.
Put $\bot\in P$ and $(n,0)\in P$ for each $n\in\om$.
As long as $x_n=0^\infty$ is true, nothing other than $\bot$ is enumerated below $(n,0)$, which becomes an atom in $P$.
If we see $x_n\not=0^\infty$ at stage $s$, we enumerate an infinite decreasing sequence $(n,0)>_P(n,s)>_P(n,s+1)>_P\dots$.

If $n$ is a witness for $x\in{\sf Truth}_{\Sigma^0_2}$, i.e., $x_n=0^\infty$, then $(n,0)$ is an atom in $P$.
Conversely, if $(n,s)$ is an atom in $P$, we must have $s=0$ and $x_n=0^\infty$; that is, $n$ is a witness for $x\in{\sf Truth}_{\Sigma^0_2}$.
\end{proof}

We also introduce a $\Sigma^0_2$-complete subobject concerning trees.

\begin{example}\label{exa:sigma-02-complete5}
Let ${\sf Tr}_2$ be the represented space of binary trees.
Consider the following subobject ${\sf Tr}_2(\geq 2)$ of ${\sf Tr}_2$:
\begin{itemize}
\item The underlying set is the set of all binary trees which have at least two infinite paths.
\item $(\sigma,\tau,p)$ is a name of $T\in{\sf Tr}_2(\geq 2)$ iff $p$ is a ${\sf Tr}_2$-name of $T$, and $\sigma$ and $\tau$ are incomparable finite strings which are extendible to infinite paths in $T$.
\end{itemize}
\end{example}

\begin{prop}\label{prop:basic-classification-sigma-02-complete}
${\sf Tr}_2(\geq 2)$ is complete w.r.t.~$\Sigma^0_2$ subobjects of $\om^\om$.
\end{prop}

\begin{proof}
It suffices to show ${\sf Truth}_{\Sigma^0_2}\leq_{\sf m}{\sf Tr}_2(\geq 2)$.
Given $x\in\om^\om$, construct a tree $T_x$ such that $0^m\in T_x$ for any $m\in\om$ and $0^n1^s\in T_x$ iff $x_n=0^\infty$ is not yet recognized when $x$ is read up to $s$.
Note that $0^\infty$ is always an infinite path through $T_x$, and $0^n1^\infty$ is an infinite path through $T_x$ iff $x_n\not=0^\infty$.
Therefore, $x\in |A|$ iff $T_x$ has at least two infinite paths, i.e., $T_x\in|{\sf Tr}_2(\geq 2)|$.

Given a witness $n$ for $x\in A$, we know $x_n\not=0^\infty$, so $(0^{n+1},0^n1)$ is an incomparable extendible pair in $T_x$; hence $r_-(n,x):=(0^{n+1},0^n1)$ is a witness for $T_x\in{\sf Tr}_2(\geq 2)$.
Conversely, if $(\sigma,\tau)$ is a witness for $T_x\in{\sf Tr}_2(\geq 2)$, either $\sigma$ or $\tau$ is of the form $0^m1^k$, and in this case, we have $x_m\not=0^\infty$; hence $r_+(\sigma,\tau,x):=m$ is a witness for $x\in A$.
\end{proof}

\subsection{Many-one degree structure}

\subsubsection{Completeness and hardness}

Let us first confirm that each of the first two levels of $\Sigma^0_2$ subobjects has a complete subobject.

\begin{theorem}\label{thm:fin-complete-unique-witness-property}
${\sf Fin}$ is complete w.r.t.~$\Sigma^0_2$ subobjects of $\om^\om$ having the unique witness property.
\end{theorem}

\begin{proof}
We first check that ${\sf Fin}$ has the unique witness property.
The idea is to use a process to search for the least witness for $x\in{\sf Fin}$ using its name.
First, we divide ${\sf Fin}$ according to what its least witness is.
Namely, let $F_n$ be a subspace of $\om^\om$ whose underlying set is:
\[|F_n|=\{x\in\om^\om:(\forall m\geq n)\;x(m)=0\ \land\ (\forall m<n)(m\leq \exists k<n)\;x(k)\not=0\}.\]

Then we claim that ${\sf Fin}\equiv\biguplus_nF_n$.
One can easily observe that $|{\sf Fin}|=|\biguplus_nF_n|$, and any name of $x\in\biguplus_nF_n$ is also a name of $x\in{\sf Fin}$.
Given a name $(n,x)$ of $x\in{\sf Fin}$, search for its least witness; that is, the least $m\leq n$ such that $x(k)=0$ for any $m\leq k\leq n$.
One can see that $x\in F_m$ for such $m$.
Then $(m,x)$ is a name of $x\in\biguplus_nF_n$.

For the completeness, assume that $A$ has the unique witness property via $A\equiv\biguplus_nA_n$.
Let $f_n$ be a witness for closedness of $A_n$; that is, $x\in A_n$ iff $f_n(x)=0^\infty$.
Given $x\in\om^\om$, by continuity of $f_n$, if $x\not\in A_n$ (i.e., $f_n(x)$ returns some nonzero value), that can be recognized after reading a finite initial segment of $x$.
Starting from $n=0$, $\varphi(x)$ keeps outputting $0$ until $f_n$ recognizes $x\not\in A_n$.
When it recognizes $x\not\in A_n$, $\varphi(x)$ outputs $1$, and then repeat the same procedure with $n+1$.
This reduction is exactly the same as the proof of the classical $\Sigma^0_2$-completeness, which shows that $x\in|\biguplus_nA_n|$ iff $\varphi(x)\in|{\sf Fin}|$.

Given a name $(n,x)$ of $x\in\biguplus_nA_n$, we have $x\in A_n$.
By uniqueness, we also have $x\not\in A_m$ for any $m<n$.
Therefore, the above procedure eventually recognizes $x\not\in A_m$ for each $m<n$, and thus, it arrives at a state waiting for $f_n$ to recognize $x\not\in A_n$ at some stage $s_n$.
Since $x\in A_n$, it is never recognized, so $\varphi(x)$ continues to output $0$ forever after stage $s_n$.
Therefore, $r_-(n,x)=(s_n,\varphi(x))$ is a name of $x\in{\sf Fin}$.

Given a name $(s,p)$ of $\varphi(x)\in{\sf Fin}$, run the above procedure for $s$ steps.
At that point, the process is waiting for $f_k$ to recognize $x\not\in A_k$ for some $k$.
If it is ever recognized, $\varphi(x)$ outputs $1$ at somewhere after the $s$th bit, which is impossible because of the assumption that $(s,p)$ is the name of $\varphi(x)\in{\sf Fin}$.
Therefore, $x\not\in A_k$ is never recognized, that is, $x\in A_k$.
Therefore, $r_+(x,(s,p))=(k,x)$ is a name of $x\in\biguplus_nA_n$.
\end{proof}

Combining with Observation \ref{obs:unique-witness-formula-chara}, this verifies our claim that ${\sf Fin}$ is complete among $\Sigma^0_2$ subobjects defined by a formula of the form $\forall n\exists m\geq n.\varphi(m,x)$, where $\varphi$ is decidable.

\begin{theorem}\label{thm:bddseq-complete-increasing-witness-property}
${\sf BddSeq}_\om$ is complete w.r.t.~$\Sigma^0_2$ subobjects of $\om^\om$ having the increasing witness property.
\end{theorem}

\begin{proof}
For each $k\in\om$, consider $B_k=\eval{x\in\om^\om:(\forall n)\;x(n)<k}$.
Then clearly $|B_k|\subseteq|B_{k+1}|$ and ${\sf BddSeq}_\om=\biguplus_{k\in\om}B_k$.
Hence, ${\sf BddSeq}_\om$ has the increasing witness property.

For the completeness, assume that $A$ has the increasing witness property via $A\equiv\biguplus_nA_n$.
Let $f_n$ be a witness for closedness of $A_n$; that is, $x\in A_n$ iff $f_n(x)=0^\infty$.
Given $x\in\om^\om$, by continuity of $f_n$, if $x\not\in A_n$, that can be recognized after reading a finite initial segment of $x$.
Starting from $n=0$, $\varphi(x)$ keeps outputting $0$ until $f_n$ recognizes $x\not\in A_n$.
When it recognizes $x\not\in A_n$, $\varphi(x)$ outputs $n$, and then repeat the same procedure with $n+1$.
This reduction is exactly the same as the proof of the classical $\Sigma^0_2$-completeness, which shows that $x\in|\biguplus_nA_n|$ iff $\varphi(x)\in|{\sf BddSeq}_\om|$.

Given a name $(k,x)$ of $x\in\biguplus_nA_n$, we have $x\in A_k$.
For the least $n\leq k$ such that $x\in A_n$, we have $x\not\in A_m$ for any $m<n$.
Therefore, the above procedure eventually recognizes $x\not\in A_m$ for each $m<n$, and thus, it arrives at a state waiting for $f_n$ to recognize $x\not\in A_n$ at some stage.
Since $x\in A_n$, it is never recognized, so $\varphi(x)$ continues to output $0$ forever after the stage.
This means, in particular, that $\varphi(x)$ only outputs values less than $n$.
Now, since $n\leq k$, $k$ gives an upper bound of $\varphi(x)$.
Therefore, $r_-(n,x)=(k,\varphi(x))$ is a name of $x\in{\sf BddSeq}_\om$.

Given a name $(b,p)$ of $\varphi(x)\in{\sf BddSeq}_\om$, note that $b$ is an upper bound of $\varphi(x)$, so let $k\leq b$ the least upper bound of $\varphi(x)$.
In particular, during the above process, $\varphi(x)$ outputs $k$ at some stage, but never output $k+1$.
This means that $f_k$ recognizes $x\not\in A_k$, but $f_{k+1}$ never recognizes $x\not\in A_{k+1}$; hence, $x\in A_{k+1}$. 
As $k+1\leq b+1$, $r_+(x,(b,p))=(b+1,x)$ is a name of $x\in\biguplus_nA_n$.
\end{proof}

Next, we show that a half $\Sigma^0_2$-hard subobject bounds these two levels.
For this purpose, we compare these notions with amalgamability.

\begin{obs}\label{obs:increasing-amalgamable}
If a ${\Sigma}^0_2$ subobject $A\mono X$ has the increasing witness property, then $A$ is amalgamable.
\end{obs}

\begin{proof}
Assume that $A$ has the increasing witness property via $A\equiv\biguplus_{n\in\om}A_n$.
Let $p$ be a name of $x\in\biguplus_nA_n$.
Given $(n_1,p_1),\dots,(n_k,p_k)$, if at least one of them, say $(n_i,p_i)$, is a correct name of $x\in\biguplus_nA_n$ then $x\in |A_{n_i}|$, and thus $x\in |A_m|$ for any $m\geq n_i$ by the increasing witness property.
Hence, $(m,p)$ is also a name of $x\in\biguplus_nA_n$.
Thus, if we put $m=\max_{j\leq k}n_j$ then $(m,p)$ is a name of $x\in\biguplus_nA_n$.
\end{proof}

\begin{prop}\label{prop:half-hard-amalgamable}
If $A\mono\om^\om$ is half $\Sigma^0_2$-hard, then $B\leq_{\sf m}A$ for any amalgamable $\Sigma^0_2$ subobject $B\mono\om^\om$.
\end{prop}

\begin{proof}
Let $\psi(x)\equiv\exists a\theta(a,x)$ be a $\Sigma^0_2$ formula defining $B$, where $\theta$ is $\Pi^0_1$.
Then, for each witness $a$ for $x\in B$, $(a,a)$ is a witness for $x\in{\sf Half}(\psi)$.
Conversely, if $(a,b)$ is a witness for $x\in{\sf Half}(\psi)$, then either $a$ or $b$ is a witness for $x\in B$.
As $B$ is amalgamable, from the information on $(a,b,x)$, one can compute a correct witness for $x\in B$.
Hence, $B\leq_{\sf m}{\sf Half}(\psi)\leq_{\sf m}A$.
\end{proof}

By Observation \ref{obs:increasing-amalgamable} and Proposition \ref{prop:half-hard-amalgamable}, we get ${\sf BddSeq}_\om\leq_{\sf m}{\sf HalfTruth}_{\Sigma^0_2}$.
We next show ${\sf HalfTruth}_{\Sigma^0_2}\leq_{\sf m}{\sf DisConn}$.

\begin{prop}\label{prop:disconn-hard}
${\sf DisConn}$ is half $\Sigma^0_2$-hard.
\end{prop}

\begin{proof}
We show that ${\sf HalfTruth}_{\Sigma^0_2}\leq_{\sf m}{\sf DisConn}$.
Given $x\in\om^\om$, we construct a graph $G=(V,E)$.
First put a special vertex $v_0\in V$.
As long as $x_n=0^\infty$ is true, we enumerate a path $(n,0)\to (n,1)\to (n,2)\to\dots$ into $E$.
If we see $x_n\not=0^\infty$ at some stage $s$, terminate the construction of this path at $(n,s)$ and then enumerate the edge $(n,s)\to v_0$ into $E$.
Note that if $x_n=0^\infty$ holds then $\{(n,i):n\in\om\}$ is a connected component in $G$; otherwise $(n,i)$ is connected to $v_0$.

Now assume that $(n,m)$ be a witness for $x\in{\sf HalfTruth}_{\Sigma^0_2}$.
Then either $x_n=0^\infty$ or $x_m=0^\infty$ holds.
If $n=m$ then $x_n=0^\infty$, so $(n,0)$ is not connected to $v_0$ by a finite path, so $\pair{(n,0),v_0}$ is a witness for $G\in{\sf DisConn}$.
Assume $n\not=m$.
If $x_n=0^\infty$ then $\{(n,i):i\in\om\}$ is a connected component in $G$ which does not contain $(m,0)$ since $n\not=m$.
If $x_n\not=0^\infty$ then we must have $x_m=0^\infty$, so $\{(m,i):i\in\om\}$ is a connected component which does not contain $(n,0)$.
Thus, $(n,0)$ and $(m,0)$ are not connected by a finite path; that is, $\pair{(n,0),(m,0)}$ is a witness for $G\in{\sf DisConn}$.

Conversely, assume that a witness for $G\in{\sf DisConn}$ is given.
First consider the case that the witness is of the form $\pair{(n,i),(m,j)}$.
If both $x_n\not=0^\infty$ and $x_m\not=0^\infty$ hold then both $(n,i)$ and $(m,j)$ are connected to $v_0$.
This is impossible.
Hence, either $x_n=0^\infty$ or $x_m=0^\infty$ holds; that is, $(n,m)$ is a witness for $x\in{\sf HalfTruth}_{\Sigma^0_2}$.
Next, if the witness is of the form $\pair{(n,i),v_0}$, then $(n,i)$ is not connected to $v_0$, so we must have $x_n=0^\infty$.
Therefore, $(n,n)$ is a witness for $x\in{\sf HalfTruth}_{\Sigma^0_2}$.
\end{proof}

\subsubsection{Separation}

We will see that these levels of $\Sigma^0_2$ subobjects have different strengths.
In other words, our goal here is to prove the following:

\begin{theorem}\label{thm:basic-sigma-2-subobject-proper}
${\sf Fin}<_{\sf m}{\sf BddSeq}<_{\sf m}{\sf HalfTruth}_{\Sigma^0_2}<_{\sf m}{\sf DisConn}<_{\sf m}{\sf Truth}_{\Sigma^0_2}$.
\end{theorem}

For this purpose, we first show a technical lemma.
An element $x\in |A|$ of a subobject $A\mono X$ is {\em splittable} if for any $\ell\in\om$ and $j,k$ there exists $x'\in|A|$ extending $x\upto\ell$ such that $j$ is no longer a witness for $x'\in A$, but if $k\not=j$ is a witness for $x\in A$, then $k$ is still a witness for $x'\in A$.

\begin{lemma}[Split Lemma]\label{lem:split-lemma}
For subobjects $A\mono X$ and $B\mono Y$, assume that $A\leq_{\sf m}B$ holds via $\varphi,r_-,r_+$ (in the sense of Example \ref{exa:computable-Levin-reducibility-b}).
If $k$ is a witness for $x\in A$ for some splittable element $x\in|A|$ then $r_+(r_-(k,x),x)=k$ holds.
\end{lemma}

\begin{proof}
Let $B\mono Y$ be given, and assume that $A\leq_{\sf m}B$ via $\varphi,r_-,r_+$.
If $j$ is a witness for $x\in A$, then $r_-(j,x)$ is a witness for $\varphi(x)\in B$.
Hence, $r_+(r_-(k,x),x)$ must be a witness for $x\in A$.
Suppose for the sake of contradiction that we have $r_+(r_-(k,x),x)=j\not=k$ for some $k$.

By continuity, we have $r_+(r_-(k,x)\upto\ell,x\upto\ell)=j$ for sufficiently large $\ell$.
Similarly, for sufficiently large $\ell'\geq\ell$, the sequence $r_-(k,x\upto\ell')$ extends $r_-(k,x)\upto\ell$.
By splittability of $x\in A$, there exists $x'\in |A|$ extending $x\upto\ell'$ such that $j$ is not a witness for $x'\in A$, but $k$ is a witness for $x'\in A$.
By our choice of $\ell'$, we have $r_+(r_-(k,x'),x')=j$.
As $k$ is a witness for $x'\in A$, $r_-(k,x')$ is a witness for $\varphi(x')\in B$, so $r_+(r_-(k,x'),x')=j$ must be a witness for $x'\in A$, which leads to a contradiction.
\end{proof}

Abstractly speaking, this roughly says that, if $x$ is splittable, $r_-(\cdot,x)$ and $r_+(\cdot,x)$ form a section-retraction pair; in particular, $r_-(\cdot,x)$ is a split mono.
We first separate ${\sf BddSeq}$ from ${\sf Fin}$.

\begin{prop}\label{prop:bddseq-amalgamable-not-uwp}
${\sf BddSeq}_\om$ does not have the unique witness property.
\end{prop}

\begin{proof}
Suppose that ${\sf BddSeq}_\om$ has the unique witness property via ${\sf BddSeq}_\om\equiv\biguplus_{n\in\om}A_n$.
We may assume that $A_n$ is a subspace of $\om^\om$.
Let $i,j$ be trackers of ${\sf BddSeq}_\om\subseteq\biguplus_nA_n$ and $\biguplus_nA_n\subseteq{\sf BddSeq}_\om$.
For any $x\in|{\sf BddSeq}_\om|$ there exists a unique $n$ such that $x\in |A_n|$.
Thus, if $(m,x)$ is a name of $x\in {\sf BddSeq}_\om$, we must have $i(m,x)=(n,x)$.
Moreover, as $(n,x)$ is a name of $x\in\biguplus_{n\in\om}A_n$, $j(n,x)$ is a name of $x\in{\sf BddSeq}_\om$, which is of the form $(b,x)$.

Of course, $(b+1,x)$ is also a name of $x\in{\sf BddSeq}_\om$, and we have $j\circ i(b+1,x)=j(n,x)=(b,x)$.
By continuity of $i$ and $j$, the first value $b$ is determined after reading $b+1$ and a finite initial segment $x\upto t$ of $x$.
However, $x\upto t$ has an extension which is bounded by $b+1$, but not by $b$.
For instance, consider an extension $y$ of $x$ such that $y(k)=b$ for any $k\geq t$.
Again, $\langle b+1,y\rangle$ is a correct name of $y\in{\sf BddSeq}_\om$, but nevertheless the first value $j\circ i(b+1,y)$ must be $b$, which is not an upper bound of $y$.
In particular, $j\circ i(b+1,y)$ cannot be a name of $y\in{\sf BddSeq}_\om$.
\end{proof}

We next separate ${\sf HalfTruth}_{\Sigma^0_2}$ from ${\sf BddSeq}$.

\begin{prop}\label{prop:half-truth-nonamalgamable}
If a $\Sigma^0_2$ subobject $A\mono\om^\om$ is amalgamable, then ${\sf HalfTruth}_{\Sigma^0_2}\not\leq_{\sf m}A$.
In particular, ${\sf HalfTruth}_{\Sigma^0_2}$ is not amalgamable.
\end{prop}

\begin{proof}
Suppose that ${\sf HalfTruth}_{\Sigma^0_2}\leq_{\sf m}A$ via $\varphi,r_-,r_+$.
Since $A\mono\om^\om$ is amalgamable, there exists a partial morphism $F$ such that if either $a$ or $a'$ is a witness for $\varphi(x)\in A$, then $F(a,a',x)$ is a witness for $\varphi(x)\in A$.
First put $x_n=0^\infty$ for each $n\in\om$.
For $x=(x_n)_{n\in\om}$, both $\{0,1\}$ and $\{2,3\}$ are witnesses for $x\in{\sf HalfTruth}_{\Sigma^0_2}$, so both $a:=r_-(\{0,1\},x)$ and $a':=r_-(\{2,3\},x)$ are witnesses for $\varphi(x)\in A$.
By amalgamability, $F(a,a',x)$ returns a witness $b$ for $\varphi(x)\in A$.
Then $r_+(b,x)$ returns a witness $\{j,k\}$ for $x\in{\sf HalfTruth}_{\Sigma^0_2}$.

Note that this $b$ is a natural number since $A$ is $\Sigma^0_2$.
Hence, the above mentioned computations are determined after reading a finite initial segment of $x$.
By changing only sufficiently large values, modify $x_j$ and $x_k$ to $x_j'$ and $x_j'$ so that $x'_j\not=0^\infty$ and $x'_k\not=0^\infty$ respectively while keeping $x_i'=0^\infty$ for each $i\not\in\{j,k\}$.
In particular, $x_i'=0^\infty$ holds for some $i<4$.
Hence, either $\{0,1\}$ or $\{2,3\}$ is a witness for $x'\in{\sf HalfTruth}_{\Sigma^0_2}$, so either $a=r_-(\{0,1\},x')$ or $a'=r_-(\{2,3\},x')$ is still a witness for $\varphi(x')\in A$.
Since only sufficiently large values of $x$ are modified to $x'$, we maintain $F(a,a',x')=b$, which must be a witness for $\varphi(x')\in A$ by amalgamability.
However, $r_+(b,x')=\{j,k\}$ is not a witness for $x\in{\sf HalfTruth}_{\Sigma^0_2}$ since our construction makes $x'_j\not=0^\infty$ and $x'_k\not=0^\infty$.
\end{proof}

We next show that ${\sf DisConn}$ is not only half $\Sigma^0_2$-hard, but is strictly stronger than any half $\Sigma^0_2$ subobject of $\om^\om$.

\begin{prop}\label{prop:disconn-hard2}
${\sf DisConn}_{\sf fun}\not\leq_{\sf m}{\sf HalfTruth}_{\Sigma^0_2}$.
\end{prop}

\begin{proof}
Assume ${\sf DisConn}_{\sf fun}\leq_{\sf m}{\sf HalfTruth}_{\Sigma^0_2}$ via $\varphi,r_-,r_+$.
A witness for ${\sf HalfTruth}_{\Sigma^0_2}$ is in the form of a pair $(a,b)$, but since it is symmetric, we may write it as $\{a,b\}$; that is, we always have $r_-(\{u,v\},\alpha),r_+(\{a,b\},\alpha)\in[\om]^2$.

Begin with the edgeless graph $\alpha\colon E\to V^2$ with $V=\N$ and $E=\emptyset$.
Note that $\alpha\in{\sf Disconn}_{\sf fun}$ is splittable since even if some finite information $\alpha\upto\ell$ of $\alpha$ is fixed, for each pair $(u,v)$ of $V$, by putting $u\edge{e}v$ for $e\in E$ larger than $\ell$, only $(u,v)$ is made connected and all other pairs are maintained to be disconnected.

Now, for three distinct vertices $u_0,u_1,u_2\in V$, $\{u_j,u_k\}$ is a witness for disconnectedness for any $j\not=k$.
Let us consider $r_-(\{u_j,u_k\},\alpha)=\{z^0_{jk},z^1_{jk}\}$, which is a witness for $\varphi(\alpha)\in{\sf HalfTruth}_{\Sigma^0_2}$.
Here, $\varphi(\alpha)$ is a sequence in $\om^\om$.
By abuse of notation, we write its $n$th term as $\varphi(n,\alpha)$.
Note that if $j\not=k$ then either $\varphi(z^0_{jk},\alpha)=0^\infty$ or $\varphi(z^1_{jk},\alpha)=0^\infty$ holds.

We are now in the situation shown in the left half of Figure \ref{figure:graph3}:
We have six candidates $\{z^0_{01},z^{1}_{01},z^0_{02},z^1_{02},z^0_{12},z^1_{12}\}$ for witnesses for $\varphi(\alpha)\in{\sf HalfTruth}_{\Sigma^0_2}$.
We analyze to which pair of vertices $r_+$ moves each pair of these six candidates.

\begin{figure}[t]
\includegraphics[width=100mm]{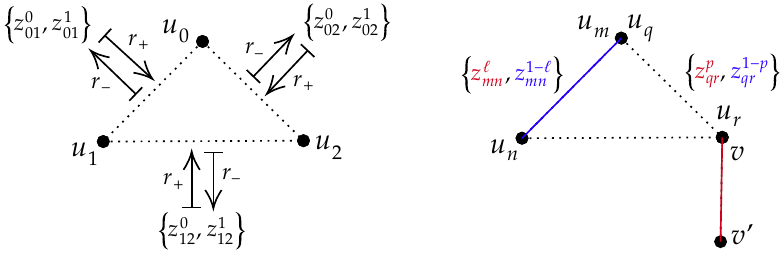}
\caption{(left) A reduction pair for ${\sf DisConn}_{\sf fun}\leq_{\sf m}{\sf HalfTruth}_{\Sigma^0_2}$ if it exists; (right) An example of our action if $r_+$ moves a pair $\{z^\ell_{mn},z^p_{qr}\}$ out of the triangle $\{u_0,u_1,u_2\}$.}\label{figure:graph3}
\end{figure}

\begin{claim}\label{claim:disconn-halftruth}
For any $\ell,m,n,p,q,r$, if $\langle \ell,\{m,n\}\rangle\not=\langle p,\{q,r\}\rangle$ then $r_+(\{z^\ell_{mn},z^p_{qr}\},\alpha)\subseteq\{u_0,u_1,u_2\}$.
\end{claim}

\begin{proof}
Otherwise, $\{v,v'\}:=r_+(\{z^\ell_{mn},z^p_{qr}\},\alpha)\not\subseteq\{u_0,u_1,u_2\}$.
If $\{m,n\}=\{q,r\}$ then $\ell\not=p$, so by Split Lemma \ref{lem:split-lemma}, we must have $\{v,v'\}=r_+(\{z^\ell_{mn},z^p_{qr}\},\alpha)=r_+(\{z^0_{mn},z^1_{mn}\},\alpha)=\{u_m,u_n\}$, which is impossible by our assumption.
Therefore, $\{m,n\}\not=\{q,r\}$.

Next, we assume that either $\varphi(z_{mn}^{1-\ell},\alpha)\not=0^\infty$ or $\varphi(z_{qr}^{1-p},\alpha)\not=0^\infty$ holds.
In the former case, by reading $\alpha$ up to a sufficiently large $t$, we recognize $\varphi(z_{mn}^{1-\ell},\alpha\upto t)\not=0^\infty$, so connect vertices $v$ and $v'$ by an edge $e\in E$ larger than $t$.
Then, $\{v,v'\}$ is not a witness for disconnectedness of the new graph $\alpha'$, but as $t$ is large, $r_+(z^\ell_{mn},z^p_{qr},\alpha')=\{v,v'\}$ is maintained.
This forces $\{z^\ell_{mn},z^p_{qr}\}$ not to be a witness for $\varphi(\alpha')\in{\sf HalfTruth}_{\Sigma^0_2}$; in particular, $\varphi(z_{mn}^{\ell},\alpha')\not=0^\infty$.
Therefore, this construction forces $\varphi(z_{mn}^i,\alpha')\not=0^\infty$ for each $i<2$.
However, $\{u_m,u_n\}$ is maintained to be disconnected, so $r_-(\{u_m,u_n\},\alpha')=\{z_{mn}^0,z_{mn}^1\}$ still returns a witness, which means that $\varphi(z_{mn}^i,\alpha')=0^\infty$ must be true for some $i<2$, which is impossible.
The same argument applies in the case that $\varphi(z_{qr}^{1-p},\alpha)\not=0^\infty$ holds.
Hence, both $\varphi(z_{mn}^{1-\ell},\alpha)=0^\infty$ and $\varphi(z_{qr}^{1-p},\alpha)=0^\infty$ are true.

By our assumption, the intersection $\{v,v'\}\cap\{u_0,u_1,u_2\}$ has at most one element, and $\{m,n\}\not=\{q,r\}$, so we have either $\{v,v'\}\cap\{u_m,u_n\}=\emptyset$ or $\{v,v'\}\cap\{u_q,u_r\}=\emptyset$.
We assume that the former holds.
The argument is the same for the latter case.

Then, consider the case that for any $i,j,k$ if $r_+(z^{1-\ell}_{mn},z^i_{jk},\alpha)$ is defined then this value is included in $\{u_m,u_n\}$.
Since both $\varphi(z_{mn}^{1-\ell},\alpha)=0^\infty$ and $\varphi(z_{qr}^{1-p},\alpha)=0^\infty$ are true, $r_+(z^{1-\ell}_{mn},z^{1-p}_{qr},\alpha)$ is defined, so this value is included in $\{u_m,u_n\}$ by our assumption.
Indeed, this value must be $\{u_m,u_n\}$ since it is a witness for disconnectedness.
Consider a graph $\alpha'$ with sufficiently large edges $v\to v'$ and $u_m\to u_n$ so that $r_+(z^{\ell}_{mn},z^p_{qr},\alpha')=\{v,v'\}$ and $r_+(z^{1-\ell}_{mn},z^{1-p}_{qr},\alpha')=\{u_m,u_n\}$ are maintained.
This makes $\{v,v'\}$ and $\{u_m,u_n\}$ connected, but the assumption $\{v,v'\}\cap\{u_m,u_n\}=\emptyset$ guarantees that no other pair is connected in $\alpha'$.
See the right half of Figure \ref{figure:graph3}.
Now $\{v,v'\}$ and $\{u_m,u_n\}$ are not witnesses for disconnectedness of $\alpha'$, so this forces $\{z^{\ell}_{mn},z^p_{qr}\}$ and $\{z^{1-\ell}_{mn},z^{1-p}_{qr}\}$ not to be witnesses for $\varphi(\alpha')\in{\sf HalfTruth}_{\Sigma^0_2}$.
In particular, we get $\varphi(z_{qr}^{i},\alpha')\not=0^\infty$ for each $i<2$.
Since $\{m,n\}\not=\{q,r\}$, the assumption $\{v,v'\}\cap\{u_m,u_n\}=\emptyset$ guarantees that $\{u_q,u_r\}$ is maintained to be disconnected in $\alpha'$ as mentioned above, so $r_-(\{u_q,u_r\},\alpha')=\{z^0_{qr},z^1_{qr}\}$ must be a witness for $\varphi(\alpha')\in{\sf HalfTruth}_{\Sigma^0_2}$.
This means $\varphi(z_{qr}^{i},\alpha')=0^\infty$ for some $i<2$, which is impossible.

Hence, there exists $i,j,k$ such that $r_+(z^{1-\ell}_{mn},z^i_{jk},\alpha)$ is defined, but the value $\{w,w'\}$ is not included in $\{u_m,u_n\}$.
Then consider a graph $\alpha'$ with sufficiently large edges $v\to v'$ and $w\to w'$ so that necessary computations are maintained.
Now $\{v,v'\}$ and $\{w,w'\}$ are not witnesses for disconnectedness of $\alpha'$, so this forces $\{z^{\ell}_{mn},z^p_{qr}\}$ and $\{z^{1-\ell}_{mn},z^i_{jk}\}$ not to be witnesses for $\varphi(\alpha')\in{\sf HalfTruth}_{\Sigma^0_2}$ as before.
In particular, we get $\varphi(z_{mn}^{i},\alpha')\not=0^\infty$ for each $i<2$.
Since $\{v,v'\}\cap\{u_m,u_n\}=\emptyset$ and $\{w,w'\}\not\subseteq\{u_m,u_n\}$, the intersection $\{v,v',w,w'\}\cap\{u_m,u_n\}$ has at most one element, so $\{u_m,u_n\}$ is maintained to be disconnected in $\alpha'$.
Thus, $r_-(\{u_m,u_n\},\alpha')=\{z_{mn}^0,z_{mn}^1\}$ must be a witness for $\varphi(\alpha')\in{\sf HalfTruth}_{\Sigma^0_2}$.
This means $\varphi(z_{mn}^{i},\alpha')=0^\infty$ for some $i<2$, which is impossible.
\end{proof}

\begin{claim}\label{claim:disconn-halftruth2}
$r_+(\{z^j_{01},z^k_{02}\},\alpha)\subseteq\{u_0,u_2\}$ for some $j,k<2$.
\end{claim}

\begin{proof}
Suppose not.
First consider the case that for each $j$ if $\varphi(z^j_{01},\alpha)=0^\infty$ then there exists $k$ such that $r_+(\{z^j_{01},z^k_{02}\},\alpha)\subseteq\{u_1,u_2\}$.
In this case, consider a graph $\alpha'$ with a sufficiently large edge $u_1\edge{e} u_2$ so that necessary computations are maintained.
Then $\{u_1,u_2\}$ is not a witness for disconnectedness of $\alpha'$, so if $\varphi(z^j_{01},\alpha)=0^\infty$, by our assumption, this forces $\{z^j_{01},z^k_{02}\}$ not to be a witness for $\varphi(\alpha')\in{\sf HalfTruth}_{\Sigma^0_2}$ as before.
In particular, we get $\varphi(z_{01}^{j},\alpha')\not=0^\infty$ and $\varphi(z_{02}^{k},\alpha')\not=0^\infty$.
If $\varphi(z^{j}_{01},\alpha)\not=\infty$, then this still holds for $\alpha'$ since $e$ above is sufficiently large.
Therefore, in any case, we obtain $\varphi(z^{j}_{01},\alpha')\not=0^\infty$ for each $j<2$.
However, $\{u_0,u_1\}$ is still a witness for disconnectedness of $\alpha'$, so $r_-(\{u_0,u_1\},\alpha')=\{z^0_{01},z^1_{01}\}$ must be a witness for $\varphi(\alpha')\in{\sf HalfTruth}_{\Sigma^0_2}$.
This means $\varphi(z^i_{01},\alpha')=0^\infty$ for some $i<2$, which is impossible.

Hence, there exists $j$ such that $\varphi(z^j_{01},\alpha)=0^\infty$ and $r_+(\{z^j_{01},z^k_{02}\},\alpha)\not\subseteq\{u_1,u_2\}$ for any $k$.
Note that $r_+(\{z^j_{01},z^k_{02}\},\alpha)$ is defined for each $k<2$ since $\varphi(z^j_{01},\alpha)=0^\infty$.
By our assumption, we also have $r_+(\{z^j_{01},z^k_{02}\},\alpha)\not\subseteq\{u_0,u_2\}$; hence by Claim \ref{claim:disconn-halftruth}, this value must be $\{u_0,u_1\}$.
In this case, consider a graph $\alpha'$ with a sufficiently large edge $u_0\to u_1$ so that necessary computations are maintained.
Then $\{u_0,u_1\}$ is not a witness for disconnectedness of $\alpha'$, which forces $\{z^j_{01},z^k_{02}\}$ not to be a witness for $\varphi(\alpha')\in{\sf HalfTruth}_{\Sigma^0_2}$ for each $k<2$.
In particular, we get $\varphi(z_{02}^{k},\alpha')\not=0^\infty$ for each $k<2$.
However, $\{u_0,u_2\}$ is still a witness for disconnectedness of $\alpha'$, so $r_-(\{u_0,u_2\},\alpha')=\{z^0_{02},z^1_{02}\}$ must be a witness for $\varphi(\alpha')\in{\sf HalfTruth}_{\Sigma^0_2}$.
This means $\varphi(z^k_{02},\alpha')=0^\infty$ for some $k<2$, which is impossible.
\end{proof}

\begin{figure}[t]
\includegraphics[width=110mm]{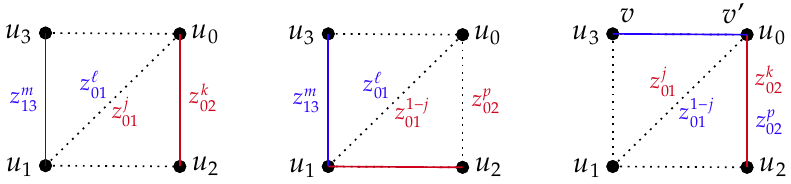}
\caption{Some examples of our actions.}\label{figure:graph4}
\end{figure}

Fix $j,k<2$ satisfying the condition in Claim \ref{claim:disconn-halftruth2}.
Now, choose a new vertex $u_3\in V$, and focus on $\{u_0,u_1,u_3\}$.
As in Claim \ref{claim:disconn-halftruth2}, one can see $r_+(\{z^\ell_{01},z^m_{13}\},\alpha)\subseteq\{u_1,u_3\}$ for some $\ell,m$.
We first assume $j\not=\ell$.
Then, consider a graph $\alpha'$ with sufficiently large edges $u_0\to u_2$ and $u_1\to u_3$ so that necessary computations are maintained.
See the leftmost part of Figure \ref{figure:graph4}.
Since $r_+(\{z^j_{01},z^k_{02}\},\alpha)\subseteq\{u_0,u_2\}$ and $r_+(\{z^\ell_{01},z^m_{13}\},\alpha)\subseteq\{u_1,u_3\}$, this forces $\varphi(z,\alpha')\not=0^\infty$ for each $z\in\{z^j_{01},z^k_{02},z^\ell_{01},z^m_{13}\}$.
In particular, $\varphi(z^i_{01},\alpha')\not=0^\infty$ for each $i<2$ since $j\not=\ell$.
Since we have taken different vertices, $\{u_0,u_2\}\cap\{u_1,u_3\}=\emptyset$.
Hence, $\{u_0,u_1\}$ is still a witness for disconnectedness of $\alpha'$, so $r_-(\{u_0,u_1\},\alpha')=\{z^0_{01},z^1_{01}\}$ must be a witness for $\varphi(\alpha')\in{\sf HalfTruth}_{\Sigma^0_2}$.
This means $\varphi(z^i_{01},\alpha')=0^\infty$ for some $i<2$, which is impossible.

Thus, we now assume $j=\ell$.
If $\varphi(z^{1-j}_{01},\alpha)\not=0^\infty$, then consider a graph $\alpha'$ with a sufficiently large edge $u_0\to u_2$ so that necessary computations are maintained.
Since $r_+(\{z^j_{01},z^k_{02}\},\alpha)\subseteq\{u_0,u_2\}$, which is maintained in $\alpha'$, this forces $\varphi(z^{i}_{01},\alpha')\not=0^\infty$ for each $i<2$.
However,  $\{u_0,u_1\}$ is still disconnected in $\alpha'$, so we get $\varphi(z^i_{01},\alpha')=0^\infty$ for some $i<2$ as before, which is impossible.
Hence, we get $\varphi(z^{1-j}_{01},\alpha)=0^\infty$.
Thus, $r_+(\{z_{01}^{1-j},z_{02}^p\},\alpha)$ is defined for each $p<2$.

If $r_+(\{z_{01}^{1-j},z_{02}^p\},\alpha)\subseteq\{u_0,u_1\}$ for each $p<2$, then consider a graph $\alpha'$ with a sufficiently large edge $u_0\to u_1$ so that necessary computations are maintained.
In particular, this forces $\varphi(z^{p}_{02},\alpha')\not=0^\infty$ for each $p<2$; however, $\{u_{0},u_2\}$ is still disconnected in $\alpha$, this is impossible as before.
Therefore, we obtain $r_+(\{z_{01}^{1-j},z_{02}^p\},\alpha)\not\subseteq\{u_0,u_1\}$ for some $p<2$.

Assume $r_+(\{z_{01}^{1-j},z_{02}^p\},\alpha)=\{v,v'\}$.
If $\{v,v'\}=\{u_1,u_2\}$, then consider a graph $\alpha'$ with sufficiently large edges $u_1\to u_2$ and $u_1\to u_3$ so that necessary computations are maintained.
See the middle part of Figure \ref{figure:graph4}.
Since $r_+(\{z^\ell_{01},z^m_{13}\},\alpha)\subseteq\{u_1,u_3\}$, this forces $\varphi(z,\alpha')\not=0^\infty$ for each $z\in\{z^{1-j}_{01},z^p_{02},z^\ell_{01},z^m_{13}\}$.
In particular, $\varphi(z^i_{01},\alpha')\not=0^\infty$ for each $i<2$ since $j=\ell$.
However, $\{u_0,u_1\}$ is still disconnected in $\alpha'$, which is impossible as in the previous argument.

If $\{v,v'\}\not=\{u_1,u_2\}$, then consider a graph $\alpha'$ with sufficiently large edges $v\to v'$ and $u_0\to u_2$ so that necessary computations are maintained.
See the rightmost part of Figure \ref{figure:graph4}.
Since $r_+(\{z^j_{01},z^k_{02}\},\alpha)\subseteq\{u_0,u_2\}$, this forces $\varphi(z,\alpha')\not=0^\infty$ for each $z\in\{z^{1-j}_{01},z^p_{02},z^j_{01},z^k_{02}\}$.
In particular, $\varphi(z^i_{01},\alpha')\not=0^\infty$ for each $i<2$.
Now note that $\{u_0,u_1\}$ is still disconnected in $\alpha'$ since $\{v,v'\}\not=\{u_1,u_2\}$.
Thus, this leads to a contradiction as before.
\end{proof}

Finally, we show that ${\sf DisConn}$ is not $\Sigma^0_2$-complete.
In fact, ${\sf DisConn}$ cannot even reduce ${\sf L}$, even though it is not amalgamable.
Here, recall that ${\sf L}$ is a $\Sigma_{\pi\cup\pi}$ subobject of $\om^\om$ introduced in Example \ref{exa:tartan-difference-hierarchy}.

\begin{prop}\label{prop:disconn-weak}
${\sf L}\not\leq_{\sf m}{\sf DisConn}$.
\end{prop}

\begin{figure}[t]
\includegraphics[width=70mm]{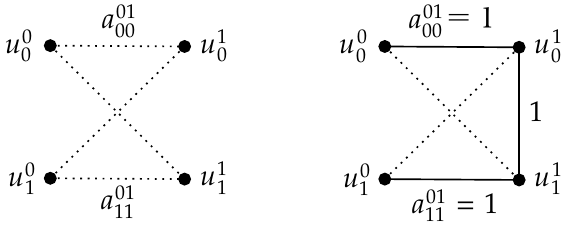}
\caption{If $\{a^{01}_{00},a^{01}_{11}\}\subseteq\{1\}$, our action forces $\{u^0_0,u^1_0\}$, $\{u^1_0,u^1_1\}$ and $\{u^1_1,u^0_1\}$ to be connected.}\label{fig:graph-proof1}
\end{figure}

\begin{proof}
Otherwise, ${\sf L}\leq_{\sf m}{\sf DisConn}$ via some $\varphi,r_-,r_+$.
Begin with $x_0=x_1=0^\infty$.
Clearly, $x:=(x_0,x_1)\in{\sf L}$ is splittable.
Then $r_-(i,x)=\{u^i_0,u^i_1\}$ is a witness for disconnectedness of the graph $G_x:=\varphi(x)$.
By Split Lemma \ref{lem:split-lemma}, we get $r_+(\{u^i_0,u^i_1\},x)=i$.
Now, consider $r_+(\{u^i_k,u^j_\ell\},x)=a^{ij}_{k\ell}$ for each $i,j,k,\ell$.
Let $t$ be a sufficiently large value such that $r_-(i,x\upto t)=\{u^i_0,u^i_1\}$ and $r_+(\{u^i_k,u^j_\ell\},x\upto t)=a^{ij}_{k\ell}$ are determined.
Similarly, if $\{u^i_k,u^j_\ell\}$ is connected in $G_x$, restrain $x\upto t$ so that $\varphi(x\upto t)$ contains a path connecting these vertices.

Assume that $\{a_{00}^{01},a_{11}^{01}\}$ is at most singleton; that is, there exists $i<2$ such that if $a_{kk}^{01}$ is defined then $a_{kk}^{01}=i$ for any $k<2$.
Then set $x_{i}'\not=0^\infty$ by changing some value of $x_{i}$ greater than $t$.
Since $r_+(\{u^{i}_0,u^{i}_1\},x')=i$, this forces $\{u^{i}_0,u^{i}_1\}$ to be connected in $G_{x'}$.
Given $k<2$, if $a_{kk}^{01}$ is defined, then $r_+(\{u^0_k,u^1_k\},x')=a_{kk}^{01}=i$, so $\{u_{k}^0,u_{k}^1\}$ must also be connected in $G_{x'}$.
If $a_{kk}^{01}$ is undefined, this means that $r_+(\{u^0_k,u^1_k\},x)$ is undefined, so $\{u_{k}^0,u_{k}^1\}$ must be connected in $G_x$.
By our choice of $t$, $\{u_{k}^0,u_{k}^1\}$ is still connected in $G_{x'}$.
As $k$ is arbitrary, now note that each of $\{u_\ell^{1-i},u_\ell^{i}\}$, $\{u^i_\ell,u^i_{1-\ell}\}$ and $\{u^i_{1-\ell},u^{1-i}_{1-\ell}\}$ is connected in $G_{x'}$; see Figure \ref{fig:graph-proof1} for $i=1$.
By connecting all of these paths, we conclude that $\{u_\ell^{1-i},u^{1-i}_{1-\ell}\}$ is connected in $G_{x'}$.
However, since $x_{1-i}'=0^\infty$, $r_-(1-i,x')=\{u^{1-i}_0,u^{1-i}_1\}$ must be a witness for disconnectedness of $G_{x'}$, which is impossible.
Hence, we must have $a_{00}^{01}\not=a_{11}^{01}$, both of which are defined.
By the same argument, one can see $a_{01}^{01}\not=a_{10}^{01}$, both of which are defined; see also the left side of Figure \ref{fig:graph-proof2}, where different values are represented by different types of lines.

\begin{figure}[t]
\includegraphics[width=70mm]{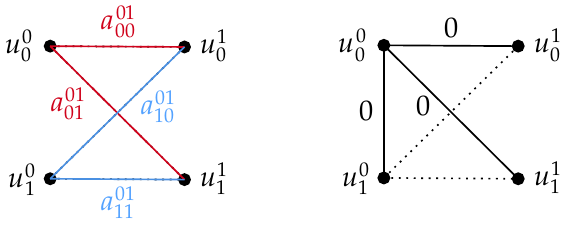}
\caption{If $a^{01}_{00}=a^{01}_{01}=0$, our action forces $\{u^0_0,u^1_0\}$, $\{u^0_0,u^1_1\}$ and $\{u^1_1,u^0_1\}$ to be connected.}\label{fig:graph-proof2}
\end{figure}

Now assume $a^{01}_{k0}=a^{01}_{k1}$ for some $k<2$.
Let $i$ be this value.
Then we get $a^{01}_{1-k,1}\not=a^{01}_{k0}=a^{01}_{k1}\not=a^{01}_{1-k,0}$, so $a^{01}_{1-k,0}=a^{01}_{1-k,1}=1-i$ holds.
Hence, we get $a^{01}_{k0}=a^{01}_{k1}=0$ for some $k$.
Then set $x_{0}'\not=0^\infty$ by changing some value of $x_{0}$ greater than $t$.
Since $r_+(\{u^0_k,u^1_\ell\},x')=a_{k\ell}^{01}=0$, this forces $\{u^0_k,u^1_0\}$ and $\{u^0_k,u^1_1\}$ to be connected in $G_{x'}$.
Similarly, since $r_+(\{u^{0}_0,u^{0}_1\},x')=0$, this also forces $\{u^{0}_0,u^{0}_1\}$ to be connected in $G_{x'}$.
This implies that $\{u^1_0,u^1_1\}$ is also connected in $G_{x'}$; see Figure \ref{fig:graph-proof2} for $k=0$.
However, since $x_{1}'=0^\infty$, $r_-(1,x')=\{u^{1}_0,u^{1}_1\}$ must be a witness for disconnectedness of $G_{x'}$, which is impossible.

Thus, $a^{01}_{k0}\not=a^{01}_{k1}$ for any $k<2$.
Then we have $a^{01}_{0\ell}\not=a^{01}_{0,1-\ell}\not=a^{01}_{1\ell}$, which implies $a^{01}_{0\ell}=a^{01}_{1\ell}$.
As in the above argument, we get $a^{01}_{0\ell}=a^{01}_{1\ell}=1$ for some $\ell$.
Then set $x_{1}'\not=0^\infty$ by changing some value of $x_{1}$ greater than $t$.
Since $r_+(\{u^0_k,u^1_\ell\},x')=a_{k\ell}^{01}=1$, this forces $\{u^0_0,u^1_\ell\}$ and $\{u^0_1,u^1_\ell\}$ to be connected in $G_{x'}$.
Similarly, since $r_+(\{u^{1}_0,u^{1}_1\},x')=1$, this also forces $\{u^{1}_0,u^{1}_1\}$ to be connected in $G_{x'}$.
This implies that $\{u^0_0,u^0_1\}$ is also connected in $G_{x'}$.
However, since $x_{0}'=0^\infty$, $r_-(0,x')=\{u^{0}_0,u^{0}_1\}$ must be a witness for disconnectedness of $G_{x'}$, which is impossible.
\end{proof}

\begin{proof}[Proof of Theorem \ref{thm:basic-sigma-2-subobject-proper}]
By Theorem \ref{thm:fin-complete-unique-witness-property} and Proposition \ref{prop:bddseq-amalgamable-not-uwp}, ${\sf Fin}$ has the unique witness property, but ${\sf BddSeq}_\om$ does not.
By Observation \ref{obs:sigma02-witnessproperty-downclosed}, if ${\sf BddSeq}_\om\leq_{\sf m}{\sf Fin}$, then ${\sf BddSeq}_\om$ must have the unique witness property, which is false; hence ${\sf BddSeq}_\om\not\leq_{\sf m}{\sf Fin}$.
By Proposition \ref{prop:simga02-uniquewitness-to-increasingwitness}, ${\sf Fin}$ has the increasing witness property; hence Theorem \ref{thm:bddseq-complete-increasing-witness-property} implies ${\sf Fin}<_{\sf m}{\sf BddSeq}_\om$.
Again by Proposition \ref{prop:bddseq-amalgamable-not-uwp}, ${\sf BddSeq}_\om$ is amalgamable, so by Propositions \ref{prop:half-hard-amalgamable} and \ref{prop:half-truth-nonamalgamable}, we get ${\sf BddSeq}_\om<_{\sf m}{\sf HalfTruth}_{\Sigma^0_2}$.
By Propositions \ref{prop:disconn-hard} and \ref{prop:disconn-hard}, we obtain ${\sf HalfTruth}_{\Sigma^0_2}<_{\sf m}{\sf DisConn}$.
Finally, by Proposition \ref{prop:disconn-weak}, ${\sf DisConn}$ is not $\Sigma^0_2$-complete; hence ${\sf DisConn}<_{\sf m}{\sf Truth}_{\Sigma^0_2}$.
\end{proof}

Finally, let us recall that we are free to choose ${\sf K_1}$, ${\sf K}_2$, or ${\sf KV}$ as our coding system.
This means that the results we have presented in this section are for the witnessed version of any of ``many-one reducibility for index sets within ${\rm Tot}$,'' ``Wadge reducibility,'' and ``effective Wadge reducibility.''

In conclusion, the notion of many-one reducibility for witnessed subsets is expected to bring new perspectives for finer analysis of definable subsets.
In this article, we have only classified $\Sigma^0_2$ subobjects, but the author and his colleagues have already started to classify higher levels of the arithmetic hierarchy, the $\Sigma^1_1$ level, and the difference hierarchy.

\begin{ack}
The author wishes to thank Koki Hashimoto for valuable discussions.
The author's research was partially supported by JSPS KAKENHI Grant Numbers 21H03392, 22K03401 and 23H03346.
\end{ack}

\bibliographystyle{plain}
\bibliography{references}

\end{document}